\documentclass[11pt, twoside]{article}

\usepackage{graphicx}
\usepackage[caption=false]{subfig}
\usepackage{amsmath}
\usepackage{amssymb,amsfonts}
\usepackage{amsthm}
\usepackage{bm}
\usepackage{mathrsfs}
\usepackage{amssymb}
\usepackage{multirow}

\newcommand{\norm}[1]{\left\lVert#1\right\rVert}

\usepackage[margin=1in]{geometry}

\usepackage{algorithm}
\usepackage{algorithmic}

\numberwithin{equation}{section}
\usepackage{multirow}
\usepackage{makecell}
\usepackage{booktabs}
\theoremstyle{definition}
\newtheorem{theorem}{Theorem}

\newtheorem{lemma}{Lemma}

\newtheorem{remark}{Remark}
\newtheorem{example}{Example}

\usepackage{cite}
\usepackage{hyperref}
\usepackage[nameinlink]{cleveref}
\newcommand{\vertiii}[1]{{\left\vert\kern-0.25ex\left\vert\kern-0.25ex\left\vert #1 
		\right\vert\kern-0.25ex\right\vert\kern-0.25ex\right\vert}}

\usepackage{fancyhdr}
\begin{document}
	
	\title{$L^\infty$ norm error estimates for  HDG methods applied to the  Poisson equation with an application to the Dirichlet boundary control problem}

\author{Gang Chen%
	\thanks{ College of  Mathematics, Sichuan University, Chengdu 610064, China (\mbox{cglwdm@uestc.edu.cn}).}
	\and
	Peter Monk%
	\thanks{Department of Mathematics Science, University of Delaware, Newark, DE, USA (\mbox{monk@udel.edu}).}
	\and
	Yangwen Zhang%
	\thanks{Department of Mathematics Science, University of Delaware, Newark, DE, USA (\mbox{ywzhangf@udel.edu}).}
}

\date{\today}

\maketitle

\begin{abstract}
 We prove quasi-optimal $L^\infty$ norm error estimates (up to logarithmic factors) for the solution of Poisson's problem by the standard Hybridizable Discontinuous Galerkin (HDG) method.  Although such estimates are available for conforming and mixed finite element methods, this is the first proof for HDG.  The method of proof is motivated by known $L^\infty$ norm estimates for mixed finite elements.  We show two applications: the first is to prove optimal convergence rates for boundary flux estimates, and the second is to prove that numerically observed convergence rates for the solution of a Dirichlet boundary control problem are to be expected theoretically.  Numerical examples show that 
the predicted rates are seen in practice.
\end{abstract}

\section{Introduction}
In this paper we derive $L^\infty$ norm estimates for the standard hybridizable discontinuous Galerkin (HDG) method applied to  a diffusion problem.  The problem is posed on a bounded convex polyhedral domain $\Omega \subset \mathbb R^2$.  We assume the data is given as follows: the
diffusivity $c\in (W^{1,\infty}(\Omega))^{2\times 2}$ is a uniformly bounded positive definite symmetric matrix-valued function, and  the   functions $f\in L^2(\Omega)$ and $g\in H^{1/2}(\partial\Omega)$.  Then we seek to approximate the solution $(u,\bm q)$ of the following elliptic system:
\begin{subequations}\label{Poisson}
	\begin{align}
	c\bm q + \nabla u &= 0  \quad  \text{in} \; \Omega \label{p1},\\
	\nabla \cdot \bm q  &=  f \quad  \text{in} \; \Omega \label{p2},\\
	u &= g \quad  \text{on}\;  \partial\Omega. \label{p3}
	\end{align}
\end{subequations}
In particular, we shall prove quasi-optimal $L^{\infty}$ error estimates (up to logarithmic factors) for the HDG approximation to $u$ and $\bm q$.  We also verify that, after a standard procedure, the post-processed solution denoted $u_h^\star$ is super-convergent in the $L^\infty$ norm.

Quasi-optimal $ L^\infty $ norm estimates on general quasi-uniform meshes for the conforming finite element method were first proved by Scott \cite{Scott_Math_Comp_1976} in 1976. The method of proof is based on weighted $L^2$ norms and  was extended in \cite{Scholz_M2AN_1977,Gastaldi_NM_1987,Duran_M2AN_1988,Duran_M2AN_1988_nonlinear,Gastaldi_M2AN_1989,wang1989asymptotic} to mixed methods for elliptic equations and in \cite{Scott_JMPA_2005,Scott_NM_2015} to the  Stokes equations. Another technique was developed in the series of papers by Schatz and Wahlbin \cite{Schatz_Math_Comp_1977,Schatz_Math_Comp_1982,Schatz_Math_Comp_1995}. They use dyadic decomposition of the domain and require local energy estimates together with sharp pointwise estimates for the corresponding components of the Green’s matrix. For smooth domains such a technique was successfully used in \cite{Chen_SINUM_2006} for mixed methods,  in \cite{Guzman_Math_Comp_2008} for discontinuous Galerkin (DG) methods and in \cite{Chen_Math_Comp_2005} for  local DG methods. This technique was also applied on a non-smooth domain for the  Stokes equations, see Guzm\'{a}n  and Leykekhman \cite{Guzman_Math_Comp_2012}.

The HDG method for elliptic equations was  devised by Cockburn et al.\ \cite{Cockburn_Gopalakrishnan_Lazarov_Unify_SINUM_2009} and was analyzed using a special projection in \cite{Cockburn_Gopalakrishnan_Sayas_Porjection_MathComp_2010}. Since there is a strong relation between the HDG method and  mixed finite element methods (see \cite{Cockburn_Gopalakrishnan_Lazarov_Unify_SINUM_2009}), it is reasonable to ask if similar $L^\infty (\Omega)$ norm estimates on general quasi-uniform meshes  could be obtained for the HDG method. To the best of our knowledge, there is no such result in the literature. In \Cref{Linftyestimates}, we give  quasi-optimal $L^\infty$ norm estimates for the flux variable $\bm q$ and scalar variable $u$ (see \Cref{main_result_Linfty_norm}). One advantage of the HDG method is that we can obtain superconvergent rates of convergence for a post processed approximation $u_h^\star$  to $u$ in the $L^2(\Omega)$ norm \cite{Cockburn_Gopalakrishnan_Lazarov_Unify_SINUM_2009}.  In \Cref{main_result_Linfty_norm}, we show that the postprocessed solution also enjoys  superconvergent rates in the $L^\infty$ norm. We present a numerical test in \Cref{example_1} (see \Cref{table_1}) to confirm our theoretical results from \Cref{main_result_Linfty_norm}. As mentioned in \cite{Melenk_IMA_2014}, we can use our $L^\infty$ norm estimates to improve $L^2(\Gamma)$ norm estimates on an interface $\Gamma$, see \Cref{L2_flux_boundary}. The numerical test in \Cref{example_1} (see \Cref{table_2}) confirms the theoretical  result from \Cref{L2_flux_boundary}. It is worthwhile to mention that a standard analysis of convergence on an interface (usually via the trace theorem) only gives a suboptimal convergence rate. 

The optimal $L^2(\Gamma)$ norm estimates on an interface $\Gamma$ or on the boundary of the domain have many applications. One example \cite{Melenk_IMA_2014} is where  some complex problems require the use of a variety
of models in different parts of the computational domain, which in turn are coupled through the normal flux across common interfaces. On the level of numerical methods, this entails a need to understand
and quantify the discretization error in the normal flux at interfaces \cite{Melenk_IMA_2014}. Another example appears in the problem of Dirichlet boundary control (DBC)  of PDEs with $L^2(\partial\Omega)$-regularization, where the
normal derivative naturally arises in the discrete optimality system. Hence, the estimation of
the error in the normal derivative plays an essential role in the error analysis of the DBC of PDEs, see \cite{Horger_CVS_2013,Melenk_SINUM_2012,Apel_MCRF_2018,Pfefferer_SINUM_2019,Winkler_NM_2020} for more details. In recent papers where HDG methods have been sucessful applied to the DBC of PDEs (\cite{MR3992054,HuShenSinglerZhangZheng_HDG_Dirichlet_control1,HuMateosSinglerZhangZhang1,HuMateosSinglerZhangZhang2,GongHuMateosSinglerZhang1}), the analysis for the control is optimal in the sense of regularity and suboptimal for other variables. Furthermore, numerical experiments show that the discrete control can achieve optimal convergence with respect to the polynomial degree if the control is smooth enough. However, the analysis in the above mentioned HDG papers is suboptimal in this situation. In \Cref{Dirichlet_Boundary_Control_Problem}, we use the improved $L^2$ norm estimates on the boundary in \Cref{L2_flux_boundary}  to obtain an optimal convergence rate for both the control and the other variables, see \Cref{main_result_optimal_control}. The numerical test in \Cref{example_2} confirms our theoretical result.

\section{HDG formulation and preliminary material}
In this section, we shall give the HDG formulation of \Cref{Poisson}, and introduce some {standard} auxiliary projections. Our main result in this section is to extend the $L^2$ norm estimates for the  auxiliary projections used in the error analysis of HDG to $L^p$ norms ($1\le p\le \infty$), see \Cref{general_lp_projection_error}. This is one essential step of the
paper. Although our final $L^\infty$ norm estimates require  the domain to be two dimensional and convex, it is worth mentioning that we do not need these restrictions in \Cref{general_lp_projection_error}. Hence, in the present  section, we assume that $\Omega\subset \mathbb R^d$, ($d=2,3$), {and do not assume convexity}.

Throughout the paper we adopt the standard notation $W^{m,p}(D)$ for Sobolev spaces on a bounded domain $D\subset \mathbb R^d$  ($d=2,3$) with norm $\|\cdot\|_{W^{m,p}(D)}$ and seminorm $|\cdot|_{W^{m,p}(D)}$:
\begin{align*}
\|u\|_{W^{m,p}(D)}^p &= \sum_{|i|\le m}\int_{D} |D^i u|^p {\rm d} \bm x,\\
|u|_{W^{m,p}(D)}^p &= \sum_{|i|= m}\int_{D} |D^i u|^p {\rm d} \bm x,
\end{align*}
where $i$ is a multi-index and $D^i$ is the corresponding partial differential operator of order $|i|$. We denote $W^{m,2}(D)$ by $H^{m}(D)$ with norm $\|\cdot\|_{H^m(D)}$ and seminorm $|\cdot|_{H^m(D)}$. Specifically, $H_0^1(D)=\{v\in H^1(D):v=0 \;\mbox{on}\; \partial D\}$.  We denote the $L^2$-inner products on $L^2(D)$ and $L^2(S)$ by
\begin{align*}
(v,w)_{D} &= \int_{D} vw  \quad \forall v,w\in L^2(D),\\
\left\langle v,w\right\rangle_{S} &= \int_{S} vw  \quad\forall v,w\in L^2(S),
\end{align*}
where $S\subset {\partial D}$. Finally, we define the space $\bm H(\text{div},\Omega) $ as 
\begin{align*}
\bm H(\text{div},\Omega) = \{\bm{v}\in [L^2(\Omega)]^d\,:\, \nabla\cdot \bm{v}\in L^2(\Omega)\}.
\end{align*}

Let $\mathcal{T}_h$ be a collection of disjoint simplices that partition $\Omega$ and satisfy the usual finite element conditions.  We denote by $\partial \mathcal{T}_h$ the set $\{\partial K: K\in \mathcal{T}_h\}$. For an element $K$ of the mesh  $\mathcal{T}_h$, let $F = \partial K \cap \partial\Omega$ denotes the boundary face of $ K $ having non-zero $d-1$ dimensional Lebesgue measure. Let $\mathcal F_h^\partial$ be the set of boundary faces and  $\mathcal F_h$ denote the set of all faces. We define the following mesh dependent norms  and spaces by
\begin{gather*}
(w,v)_{\mathcal{T}_h} = \sum_{K\in\mathcal{T}_h} (w,v)_K,   \quad\quad\quad\quad\left\langle \zeta,\rho\right\rangle_{\partial\mathcal{T}_h} = \sum_{K\in\mathcal{T}_h} \left\langle \zeta,\rho\right\rangle_{\partial K},\\
H^1(\mathcal T_h) = \prod_{K\in \mathcal T_h} H^1(K), \quad\quad\quad\quad  L^2 (\partial\mathcal T_h) =  \prod_{K\in \mathcal T_h} L^2(\partial K).
\end{gather*}

Let $\mathcal{P}^k(D)$ denote the set of polynomials of degree at most $k$ on a domain $D$.  We introduce the discontinuous finite element spaces used in the HDG method as follows:
\begin{align*}
\bm{V}_h  &:= \{\bm{v}_h\in [L^2(\Omega)]^d: \bm{v}_h|_{K}\in [\mathcal{P}^k(K)]^d, \forall K\in \mathcal{T}_h\},\\
{W}_h  &:= \{{w}_h\in L^2(\Omega): {w}_h|_{K}\in \mathcal{P}^{k}(K), \forall K\in \mathcal{T}_h\},\\
\widehat W_h  &:= \{\widehat w_h \in L^2({\mathcal F_h}): \widehat w_h|_{F}\in \mathcal{P}^k(F), \forall F\in \mathcal F_h; \widehat w_h|_F=0, \forall F\in \mathcal F_h^\partial \}.
\end{align*}

\subsection{HDG formulation}
To simplify the presentation, we assume the Dirichlet boundary condition is homogeneous, i.e., $g=0$. Then the HDG method of Cockburn et al.\ \cite{Cockburn_Gopalakrishnan_Lazarov_Unify_SINUM_2009} seeks the flux ${\bm{q}}_h \in \bm{V}_h $, the scalar variable $ u_h  \in W_h $ and  its numerical trace $ \widehat{u}_h\in \widehat W_h$ satisfying
\begin{subequations}\label{HDG_discrete2}
	\begin{align}
	(c\bm{q}_h, \bm{v}_h)_{{\mathcal{T}_h}}- (u_h, \nabla\cdot \bm{v}_h)_{{\mathcal{T}_h}}+\langle \widehat u_h, \bm{v}_h\cdot \bm{n} \rangle_{\partial{{\mathcal{T}_h}}} &=0, \label{HDG_discrete2_a}\\
	-(\bm{q}_h, \nabla w_h)_{{\mathcal{T}_h}}
	+\langle\widehat{\bm{q}}_h\cdot\bm{n}, w_h \rangle_{\partial{{\mathcal{T}_h}}}&=(f, w_h)_{{\mathcal{T}_h}}, \label{HDG_discrete2_b}\\
	\langle\widehat{\bm{q}}_h\cdot \bm{n}, \widehat w_h \rangle_{\partial\mathcal{T}_{h}}&=0, \label{HDG_discrete2_e}
	\end{align}
	for all $(\bm{v}_h,w_h,\widehat w_h)\in \bm{V}_h\times W_h\times \widehat W_h$.
	The numerical traces on $\partial\mathcal{T}_h$ are defined by \cite{Cockburn_Gopalakrishnan_Sayas_Porjection_MathComp_2010}
	\begin{align}
	\widehat{\bm{q}}_h\cdot \bm n =\bm q_h\cdot\bm n+\tau (u_h-\widehat u_h)   \quad \mbox{on} \; \partial \mathcal{T}_h, \label{HDG_discrete2_h}
	\end{align}
\end{subequations}
where the stabilization parameter $\tau\in L^\infty({\cal F}_h)$ is uniformly positive and bounded. For simplicity, we consider the stabilization function $\tau $ to be constant on the boundary of each element. 

After computing the solution $(\bm q_h, u_h, \widehat u_h) $ of \eqref{HDG_discrete2}, we can use the following element-by-element postprocessing to find $u_{h}^\star|_K \in \mathcal P^{k+1}(K)$ such that for all $(z_h, w_h)\in \mathcal [\mathcal P^{k+1}(K)]^{\perp}\times\mathcal{P}^{0}(K) $ 
\begin{subequations}\label{post_process_1}
	\begin{align}
	(\nabla u_{h}^{\star},\nabla z_h)_K&=-(c\bm q_{h},\nabla z_h)_K,\label{post_process_1_a}\\
	(u_{h}^{\star},w_h)_K&=(u_h,w_h)_K,\label{post_process_1_b}
	\end{align}
\end{subequations}
where $\mathcal [\mathcal P^{k+1}(K)]^{\perp} = \{z_h\in \mathcal P^{k+1}(K)\,:\, (z_h,1)_K = 0\} $. 

To shorten lengthy equations, we define the following HDG bilinear form $ \mathscr B: \bm H^1(\mathcal T_h) \times H^1(\mathcal T_h)\times L^2 (\partial\mathcal T_h)\times  \bm H^1(\mathcal T_h) \times H^1(\mathcal T_h)\times L^2 (\partial\mathcal T_h)\to \mathbb R$ by
\begin{equation}\label{def_B}
\begin{split}
\mathscr  B( \bm q,u,\widehat u;\bm v,w,\widehat w)&=(c\bm{q}, \bm v)_{{\mathcal{T}_h}}- (u, \nabla\cdot \bm v)_{{\mathcal{T}_h}}+\langle \widehat{u}, \bm v\cdot \bm{n} \rangle_{\partial{{\mathcal{T}_h}}}\\
&\quad- (\nabla\cdot\bm{q},  w)_{{\mathcal{T}_h}}-\langle  \tau ( u - \widehat u),  w-\widehat w \rangle_{\partial{{\mathcal{T}_h}}} + \langle \bm q\cdot\bm n, \widehat w \rangle_{\partial \mathcal T_h}.
\end{split}
\end{equation}

By the definition of $\mathscr  B$ in \eqref{def_B},  we can rewrite the HDG formulation of system \eqref{HDG_discrete2}, as follows: find $({\bm{q}}_h,u_h,\widehat u_h)\in \bm V_h\times W_h\times \widehat W_h$  such that
\begin{align}\label{HDG_full_discrete}
\mathscr B (\bm q_h,u_h,\widehat u_h;\bm v_h,w_h,\widehat w_h) =  - (f,w_h)_{\mathcal T_h}
\end{align}
for all $\left(\bm{v}_h,w_h,\widehat w_h\right)\in \bm V_h\times W_h\times \widehat W_h$. Moreover, the exact solution $(\bm q, u)$ also satisfies equation \eqref{HDG_full_discrete}, i.e.,
\begin{align}\label{HDG_exact}
\mathscr B (\bm q,u, u;\bm v_h,w_h,\widehat w_h) =  - (f,w_h)_{\mathcal T_h}
\end{align}
for all $ (\bm v_h,w_h,\widehat w_h) \in \bm V_h\times W_h\times \widehat W_h$.

From \cite[Lemma 2]{ChenMonkZhang1} we recall the following stability result. 
\begin{lemma}\label{energy_norm}
	For any $ (\bm q_h,u_h,\widehat u_h) \in \bm V_h\times W_h\times \widehat W_h$, we have
	\begin{align*}
	\mathscr B (\bm q_h,u_h,\widehat u_h;\bm q_h,-u_h,-\widehat u_h)= (c\bm{q}_h, \bm{q}_h)_{{\mathcal{T}_h}}+\langle \tau( u_h- \widehat u_h),  u_h- \widehat u_h\rangle_{\partial{{\mathcal{T}_h}}}.
	\end{align*}
\end{lemma}

The following lemma shows that the bilinear form $\mathscr B$ is symmetric and is proved by integration by parts. We do not provide details.
\begin{lemma}\label{eq_B}
	For any $ (\bm q,u,\widehat u; \bm v,w,\widehat w) \in \bm H^1(\mathcal T_h) \times H^1(\mathcal T_h)\times L^2 (\partial\mathcal T_h)\times  \bm H^1(\mathcal T_h) \times H^1(\mathcal T_h)\times L^2 (\partial\mathcal T_h)$, we have
	\begin{align}\label{commute}
	\mathscr B (\bm q,u,\widehat u;\bm v,w,\widehat w)  =  	\mathscr B (\bm v,w,\widehat w;\bm q,u,\widehat u).
	\end{align}
\end{lemma}

\subsection{Preliminary material}\label{sec:Projectionoperator}
Recall the HDG projection $\Pi_h(\bm q, u):=(\bm \Pi_V \bm q, \Pi_W u )$   (see \cite[equation (2.1a)-(2.1c)]{Cockburn_Gopalakrishnan_Sayas_Porjection_MathComp_2010}), that satisfies the following equations:
\begin{subequations}\label{HDG_projection_operator}
	\begin{align}
	(\bm\Pi_V\bm q,\bm v_h)_K&=(\bm q,\bm v_h)_K,\qquad\qquad\qquad\qquad \forall  \ \bm v_h\in[\mathcal{\bm P}^{k-1}(K)]^d,\label{projection_operator_1}\\
	(\Pi_Wu, w_h)_K&=(u, w_h)_K,\qquad\qquad\qquad\qquad \forall   \ w\in \mathcal P^{k-1}(K),\label{projection_operator_2}\\
	\langle\bm\Pi_V\bm q\cdot\bm n+\tau  \Pi_W u,\widehat w_h\rangle_{F} &= \langle\bm q\cdot\bm n+ \tau  u,\widehat w_h\rangle_{F},\qquad\qquad~\;\forall \  \widehat w_h\in \mathcal P^{k}(F),\label{projection_operator_3}
	\end{align}
	for all faces $F$ of the simplex $K$. If $k=0$, then \eqref{projection_operator_1} and \eqref{projection_operator_2} are vacuous and $\Pi_h$ is defined solely by  \eqref{projection_operator_3}. Note that although we denoted the first component of the projection by $\bm\Pi_V\bm q$, it depends not just on $\bm q$, but  on both $\bm q$ and $u$, as we see from \eqref{HDG_projection_operator}. The same is true for $\Pi_W u$. Hence the notation $(\bm \Pi_V \bm q, \Pi_W u )$ for $\Pi_h(\bm q, u)$ is somewhat misleading, but its convenience outweighs this disadvantage.
\end{subequations}

It is worthwhile  mentioning that the domain of the projection $\Pi_h$ is a subspace of $[L^2(\Omega)]^d\times L^2(\Omega)$ on which the right hand sides of \eqref{HDG_projection_operator} are well defined. \emph{We do not require that  the two components $(\bm q, u)$ satisfy the equation \eqref{p1}}.

The well-posedness of $(\bm \Pi_V,\Pi_W)$ and its approximation properties are given in the following \Cref{pro_error}. The proof  can be found in \cite[ Appendix]{Cockburn_Gopalakrishnan_Sayas_Porjection_MathComp_2010}.
\begin{lemma}\label{pro_error}
	Suppose $\tau|_{\partial K}$ is a positive constant.  Then the system \eqref{HDG_projection_operator} is uniquely solvable for $\bm{\Pi}_V\bm{q}$ and $\Pi_W u$. Furthermore, there is a constant $C$ independent of $K$ and $\tau$ such that
	\begin{subequations}
		\begin{align}		
		\|{{\Pi}_W}{u}-u\|_{ L^2(K)}&\leq Ch_{K}^{\ell_{u}+1}|{u}|_{H^{\ell_u+1}(K)}+C\frac{h_{K}^{\ell_{\bm q}+1}}{\tau_K^{\max}}{| \nabla\cdot \bm{q}|}_{\bm H^{\ell_{\bm q}}(K)},\label{Proerr_u}\\
		\|{\bm{\Pi}_V}\bm{q}-\bm q\|_{\bm L^2(K)} &\leq Ch_{K}^{\ell_{\bm q}+1}|\bm{q}|_{\bm H^{\ell_{\bm q}+1}(K)}+Ch_{K}^{\ell_u+1} \tau_K^\star {|u|}_{H^{\ell_u+1}(K)},\label{Proerr_q}
		\end{align}
	\end{subequations}
	for $\ell_u, \ell_{\bm q}\in [0,k]$. Here $\tau_K^\star: = \max \tau|_{\partial K\backslash F^\star}$,  where $F^\star$ is a face of $K$ at which $\tau|_{\partial K}$ is maximum.
\end{lemma}

Besides the projections $\bm \Pi_V$ and $\Pi_W$, in the analysis we also need to  introduce the standard local $L^2$ projection operators $\Pi_\ell^o: L^2(K)\to \mathcal P^{\ell}(K)$ and $\Pi_k^\partial: L^2(F)\to \mathcal P^{k}(F)$ satisfying
\begin{subequations}
	\begin{align}
	(\Pi_\ell^o w, w_h)_K &= (w,w_h)_K,\qquad \forall \; w_h\in \mathcal P^\ell(K),\label{L2_do}\\
	\langle \Pi_k^\partial   w, \widehat w_h \rangle_F &= \langle w,  \widehat w_h  \rangle_F,\qquad \forall \;   \widehat w_h \in \mathcal P^{k}(F).\label{L2_edge}
	\end{align}
\end{subequations}
We use $\bm \Pi_\ell^o$ to denote the local vector $L^2$ projection operator, the definition componentwise is the same as local scalar  $L^2$ projection operator.  The next lemma gives the approximation properties of $\Pi_\ell^o$ and its proof can be found in \cite[Theorem 3.3.3, Theorem 3.3.4]{book2}.
\begin{lemma}\label{approximation_of_L2}
	Let $\ell\ge 0$ be an integer and  $\rho\in [1, +\infty]$. If $(\ell+1)\rho<d$, then we require $d, \rho$ and $\ell$ to also satisfy $2\le \frac{d\rho}{d -(\ell+1)\rho }$. For $j\in \{0, 1 ,\ldots, \ell+1\}$, if $s_j$ satisfies
	\begin{align}\label{index_s}
	\begin{cases}
	\rho \le s_j \le  \frac{d\rho}{d -(\ell+1-j)\rho } &(\ell+1-j)\rho <d,\\
	\rho \le s_j  < \infty&(\ell+1-j)\rho =d,\\
	\rho \le s_j \le \infty &(\ell+1-j)\rho >d,
	\end{cases}
	\end{align}
	then there exists a constant $C$ which is independent of $K$ such that 
	\begin{subequations}
		\begin{align}
		\| \nabla^j(\Pi_\ell^o u - u) \|_{L^{s_j}(K)} &\le C h_K^{\ell+1-j+\frac{d}{s_j} - \frac{d}{\rho}}|u|_{W^{\ell+1, \rho}(K)},\label{Pro_jec_1}\\
		\| \nabla^j(\Pi_\ell^o u - u) \|_{L^{s_j}(\partial K)}&\le C h_K^{\ell+1-j+\frac{d-1}{s_j} - \frac{d}{\rho}}|u|_{W^{\ell+1, \rho}(K)}.\label{Pro_jec_2}
		\end{align}	
	\end{subequations}
\end{lemma}
In the analysis, we also need the following {standard} inverse inequality~{\cite[Theorem 3.4.1]{book2}}.
\begin{lemma}[Inverse inequality]\label{Inverse_inequality}
	Let $k\ge 0$ be an integer, $\mu, \rho\in [1, +\infty]$ and $\eta>0$, then there exists $C$ depend on $k, \mu, \rho, d$ and $\eta$ such that 
	\begin{align}\label{inverse}
	|u_h|_{t, \mu, K} \le Ch_K^{\frac{d}{\mu} - \frac{d}{\rho}-t+s}	|u_h|_{s,\rho, K}, \quad \forall u_h\in \mathcal P_k(K), \quad t\ge s.
	\end{align}
\end{lemma}

In the analysis, not only the $L^2$ approximation properties of $(\bm \Pi_V, \Pi_W)$ are important,  for us, but the $L^{\infty}$ approximation of these projection operators plays an essential role. We provide these estimates in the next theorem.

\begin{theorem}\label{general_lp_projection_error}
	Let $k\ge 0$ be an integer and  $\rho\in [1, +\infty]$. If $(k+1)\rho<d$, then we need $d, \rho$ and $k$ to  satisfy $2\le \frac{d\rho}{d -(k+1)\rho }$. For $j\in \{0, 1 ,\ldots, k+1\}$, if $s_j$ satisfies \eqref{index_s}, then
	\begin{subequations}
		\begin{align}
		\| \Pi_{W} u - u \|_{L^{s_j}(K)}&\le  \frac{C}{\tau}{h_{K}^{k+1+\frac{d}{s_j} - \frac{d}{\rho}}}{|\nabla\cdot \bm{q}|}_{W^{k,\rho}(K)} +C h_{K}^{k+1+\frac{d}{s_j} - \frac{d}{\rho}}|{u}|_{W^{k+1,\rho}(K)},\label{error_pi_u1}\\
		\| \bm \Pi_{V} \bm q  - \bm q \|_{\bm L^{s_j}(K)} &\le Ch_{K}^{k+1+\frac{d}{s_j} - \frac{d}{\rho}}|\bm{q}|_{\bm W^{k+1,\rho}(K)}+Ch_{K}^{k+1+\frac{d}{s_j} - \frac{d}{\rho}} {|u|}_{ W^{k+1,\rho}(K)}\label{error_pi_q1}.
		\end{align}	
	\end{subequations}
\end{theorem}
\begin{proof}
	First, we prove \eqref{error_pi_u1}.  In the proof of \cite[Proposition A.2.]{Cockburn_Gopalakrishnan_Sayas_Porjection_MathComp_2010} we have 
	\begin{align}\label{from_Bernardo}
	\begin{split}
	\|\Pi_W u - \Pi_k^o u\|_{L^2(K)}& \le C \frac{h_K}{\tau} \left( \|\nabla\cdot \bm q -  \Pi_{k-1}^o (\nabla\cdot \bm q )\|_{L^2(K)}  \right) \\
	&\quad + C \left( \| u -\Pi_k^o u\|_{L^2(K)} + h_K\| \nabla (u -\Pi_k^o u)\|_{L^2(K)} \right).
	\end{split}
	\end{align}
	Then, using the local inverse estimate in \Cref{Inverse_inequality},
	\begin{align*}
	\|\Pi_W u -  u\|_{L^{s_j}(K)} &\le 	\|\Pi_W u -   \Pi_k^o u\|_{L^{s_j}(K)}  + 	\|  \Pi_k^o u - u\|_{L^{s_j}(K)} \\
	&\le 	Ch_K^{\frac{d}{s_j} - \frac{d}{2}}\|\Pi_W u -   \Pi_k^o u\|_{L^{2}(K)}  + 	\|  \Pi_k^o u - u\|_{L^{s_j}(K)}&\textup{by }\eqref{inverse}\\
	&\le 	Ch_K^{\frac{d}{s_j} - \frac{d}{2}} \left(\frac{h_K}{\tau}\|\nabla\cdot \bm q -  \Pi_{k-1}^o (\nabla\cdot \bm q )\|_{L^2(K)} \right.\\ 
	&\qquad \qquad \quad \left. +  \| u -\Pi_k^o u\|_{L^2(K)} + h_K\| \nabla (u -\Pi_k^o u)\|_{L^2(K)}\right) &\textup{by }\eqref{from_Bernardo}\\
	&\quad +	\|  \Pi_k^o u - u\|_{L^{s_j}(K)} \\
	&\le \frac{C}{\tau}{h_{K}^{k+1+\frac{d}{s_j} - \frac{d}{\rho}}}{|\nabla\cdot \bm{q}|}_{W^{k,\rho}(K)} +C h_{K}^{k+1+\frac{d}{s_j} - \frac{d}{\rho}}|{u}|_{W^{k+1,\rho}(K)}&\textup{by }\eqref{Pro_jec_1}.
	\end{align*}

	Next, we prove  \eqref{error_pi_q1}. Because we do not have an estimate like \eqref{from_Bernardo} for $\bm \Pi_V$,  we  introduce the single face HDG projection $\bm B_V$  defined on sufficiently smooth vector functions $\bm q$ such that $\bm B_V \bm q \in [\mathcal P^k(K)]^d$ satisfies \cite[equation (3.10)]{Cockburn_JSC_2007}:
	\begin{subequations}\label{HDG_projection_operatorB}
		\begin{align}
		(\bm B_V\bm q,\bm v_h)_K&=(\bm q,\bm v_h)_K,\qquad\qquad\quad~ \forall  \ \bm v_h\in[\mathcal{\bm P}^{k-1}(K)]^d,\label{projection_operator_B1}\\
		\langle\bm B_V\bm q\cdot\bm n,\mu_h\rangle_{F_i} &= \langle\bm q\cdot\bm n,\mu_h\rangle_{F_i}  \qquad\qquad~\;\forall \  \mu_h\in \mathcal P^{k}(F_i),i=1\ldots, d.\label{projection_operator_B2}
		\end{align}
	\end{subequations}
	By  \cite[equation (3.13) of Lemma 3.2, Lemma 3.3]{Cockburn_JSC_2007}  we have 
	\begin{subequations}
		\begin{align}
		\|\bm B_V\bm q\|_{\bm L^2(K)} &\le \|\bm q\|_{\bm L^2(K)} +h_K^{1/2}\|\bm q\cdot \bm n\|_{ L^2(\partial K)},\label{eds1}\\
		\|\bm B_V\bm q - \bm q\|_{\bm H^s(K)} &\le C h_K^{k+1-s}|\bm q|_{\bm H^{k+1}(K)},\label{eds2}
		\end{align}
	\end{subequations}
	for $0\le s \le k+1$. Note that $\bm B_V \bm q_h = \bm q_h$ for any $\bm q_h\in [\mathcal P^k(K)]^d$. Then, using the local inverse estimate in \Cref{Inverse_inequality},
	\begin{align*}
	\|\bm B_V \bm q - \bm q\|_{\bm L^{s_j} (K)}&\le \|\bm B_V \bm q - \bm\Pi_k^o\bm q\|_{\bm L^{s_j} (K)} + \| \bm\Pi_k^o\bm q - \bm q\|_{\bm L^{s_j} (K)}\\
	&\le Ch_K^{\frac{d}{s_j} -\frac{d}{2}} \|\bm B_V \bm q - \bm\Pi_k^o\bm q\|_{\bm L^{2} (K)} + \| \bm\Pi_k^o\bm q - \bm q\|_{\bm L^{s_j} (K)}&\textup{by }\eqref{inverse}\\
	&= Ch_K^{\frac{d}{s_j} -\frac{d}{2}} \|\bm B_V (\bm q - \bm\Pi_k^o\bm q)\|_{\bm L^{2} (K)} + \| \bm\Pi_k^o\bm q - \bm q\|_{\bm L^{s_j} (K)}\\
	&\le  Ch_K^{\frac{d}{s_j} -\frac{d}{2}} \left(\|\bm q - \bm\Pi_k^o\bm q\|_{\bm L^{2} (K)} + h_K^{1/2} \|\bm q - \bm\Pi_k^o\bm q\|_{\bm L^2(\partial K)} \right)&\textup{by }\eqref{eds1}\\
	&\quad  + \| \bm\Pi_k^o\bm q - \bm q\|_{\bm L^{s_j} (K)}\\
	&\le Ch_{K}^{k+1+\frac{d}{s_j} - \frac{d}{\rho}}|\bm{q}|_{\bm W^{k+1,\rho}(K)}&\textup{by }\eqref{Pro_jec_1}.
	\end{align*}
	where in the last inequality we  used \eqref{Pro_jec_1} and  \eqref{Pro_jec_2}. In the proof of \cite[Proposition A.3.]{Cockburn_Gopalakrishnan_Sayas_Porjection_MathComp_2010} we find that  
	\begin{align}\label{def3}
	\|\bm B_V \bm q - \bm \Pi_V\bm q\|_{\bm L^{2} (K)} \le C\left(h_K^{1/2}\|u - \Pi_k^o u\|_{ L^{2} (K)} +  \|\Pi_W u - \Pi_k^o u\|_{ L^{2} (K)}\right).
	\end{align}
	Then, again using the local inverse estimate in \Cref{Inverse_inequality},
	\begin{align*}
	\|\bm \Pi_V \bm q - \bm q\|_{\bm L^{s_j} (K)}&\le  \|\bm B_V \bm q - \bm \Pi_V\bm q\|_{\bm L^{s_j} (K)} + \|   \bm B_V\bm q  - \bm q   \|_{\bm L^{s_j} (K)}\\
	&\le Ch_K^{\frac{d}{s_j} -\frac{d}{2}} \|\bm B_V \bm q - \bm\Pi_V \bm q\|_{\bm L^{2} (K)} + \| \bm B_V\bm q - \bm q\|_{\bm L^{s_j} (K)}&\textup{by }\eqref{inverse}\\
	&\le  Ch_K^{\frac{d}{s_j} -\frac{d}{2}}\left(h_K^{1/2}\|u - \Pi_k^o u\|_{ L^{2} (\partial K)} +  \|\Pi_W u - \Pi_k^o u\|_{ L^{2} (K)}\right)&\textup{by }\eqref{def3}\\
	&\quad  + \| \bm B_V\bm q - \bm q\|_{\bm L^{s_j} (K)}\\
	&\le Ch_{K}^{k+1+\frac{d}{s_j} - \frac{d}{\rho}}|\bm{q}|_{W^{k+1,\rho}(K)} + Ch_{K}^{k+1+\frac{d}{s_j} - \frac{d}{\rho}}|{u}|_{W^{k+1,\rho}(K)},
	\end{align*}
	where in the last inequality we split $\Pi_W u - \Pi_k^o u = \Pi_W u -  u + u-\Pi_k^o u$ and used \eqref{Pro_jec_1}, \eqref{Pro_jec_2} and \eqref{error_pi_q1}.
\end{proof}

\section{$L^\infty$ norm estimates}\label{Linftyestimates}
In the rest of this paper, we restrict the domain $\Omega$ to two dimensional space, i.e., $d=2$. Furthermore, we assume:
\begin{itemize}
	\item[(\bf{A})] The domain is convex and the triangular mesh $\mathcal T_h$ is quasi-uniform.
\end{itemize}

Now, we  state the main result of our paper:
\begin{theorem}\label{main_result_Linfty_norm}
	Let $(\bm q, u)$ and $(\bm q_h, u_h,\widehat u_h)$ be the solution of \eqref{Poisson} and \eqref{HDG_discrete2}, respectively.  We assume that (A) holds.   First, if  $u\in L^\infty(\Omega)$, $\bm q\in \bm L^\infty(\Omega)$ and $f\in L^2(\Omega)$, then we have the following stability bounds:
	\begin{subequations}\label{maxmimu_norm}
		\begin{align}
		\|u_h\|_{L^\infty(\Omega)} &\le   \|u\|_{L^\infty(\Omega)}   + Ch^{\min\{k,1\}}\|f\|_{L^2(\Omega)} &\textup{ for all } k\ge 0, \label{maxmimu_norm_u}\\
		\|\bm q_h\|_{L^\infty(\Omega)} &\le   \|\bm q\|_{\bm L^\infty(\Omega)}  + C\|f\|_{L^2(\Omega)}&\textup{ for all } k\ge 1.\label{maxmimu_norm_q}
		\end{align}	
		Second, if $(\bm q, u)\in \bm W^{k+1,\infty}(\Omega) \times  W^{k+1,\infty}(\Omega)$, then we have the following error estimates:
		\begin{align}
		\|\bm q- \bm q_h\|_{\bm L^\infty(\Omega)} &\le Ch^{k+1}(|\log h |^{1/2} +1 )(|\bm q|_{\bm W^{k+1,\infty}(\Omega)} + |u|_{ W^{k+1,\infty}(\Omega)}),&\textup{ for all } k\ge 1,\label{thmL_infty_q}\\
		\|u- u_h\|_{L^\infty(\Omega)} &\le Ch^{k+1} (|\log h | +1 )(|\bm q|_{\bm W^{k+1,\infty}(\Omega)} + |u|_{ W^{k+1,\infty}(\Omega)})&\textup{ for all } k\ge 1\label{thmL_infty_u}.
		\end{align}
		Furthermore, if $(\bm q, u)\in \bm W^{k+1,\infty}(\Omega) \times  W^{k+2,\infty}(\Omega)$, then we have the following error estimate for the postprocessed solution: 
		\begin{align}\label{thmL_infty_ustar}
		\|u - u_{h}^{\star}\|_{L^\infty(\Omega)} \le C  h^{k+2}(1+|\log h|) (|\bm q|_{\bm W^{k+1,\infty}(\Omega)} + |u|_{ W^{k+2,\infty}(\Omega)}) &\textup{ for all } k\ge 1,
		\end{align}
		where $u_h^\star $ was defined in \eqref{post_process_1}.
	\end{subequations}
\end{theorem}
The remainder of this section will be devoted to proving the above result.

\subsection{Proof of \Cref{main_result_Linfty_norm}} 
We start the proof of \Cref{main_result_Linfty_norm} by defining suitable regularized Green's functions.
We follow the notation of Girault, Nochetto and Scott \cite{Scott_JMPA_2005} to define by $\delta_\star\ge 0$ the usual mollifier in $\mathcal D(\mathbb R^2)$ such that $\textup{supp}(\delta_\star)\subset B(0,1)$ and $\int_{\mathbb R^2} \delta_\star(x) {\rm d}\bm x=1$. Then for any point $\bm x_0\in \Omega$ and real number $\rho_0>0$ such that the ball $B(\bm x_0, \rho_0)$ is contained in $\Omega$, we define the mollifier by
\begin{align}
{\delta(\bm x)} = \frac{1}{\rho_0^2}\delta_\star\left(\frac{\bm x - \bm x_0}{\rho_0}\right).
\end{align}

\begin{lemma}[{\rm\cite[Lemma 1.1]{Scott_JMPA_2005}}]\label{delta_function_1}
	Suppose the triangular mesh $\mathcal T_h$ is quasi-uniform. Let  $\varphi_h$ be a polynomial in $\mathcal P^k$ on each $K$,  $\bm x_M$ be a point of $\bar\Omega$ where $|\varphi_h(\bm x)|$ attains its maximum,  $K$ be an element containing $\bm x_M$ and  $B \subset K$ be the {disk of} radius $\rho_K$ inscribed in $K$. Then there exists a smooth function $\delta_M$ supported in $B$ such that
	\begin{subequations}
		\begin{gather}
		\int_{\Omega} \delta_M ~{\rm d}\bm x = 1,\label{delta_function_2}\\
		\|\varphi_h\|_{L^\infty(\Omega)} = \left|\int_B\delta_M\varphi_h~{\rm d}\bm x\right|,\label{delta_function_3}
		\end{gather}
		and for any number $t$ with $1<t\le \infty$, there exists a constant $C$, such that
		\begin{align}
		\|\delta_M\|_{L^t(B)}&\le h^{2/t-2},\label{delta_function_4}\\
		\|\nabla \delta_M\|_{L^t(B)}&\le h^{2/t-3}.\label{delta_function_5}
		\end{align}
	\end{subequations}
\end{lemma}	
\begin{proof}
	The proof of \eqref{delta_function_2}-\eqref{delta_function_4} can be found in \cite[Lemma 1.1]{Scott_JMPA_2005} where it is shown that  there exists polynomial $P_M\in \mathcal P^k(K)$ such that 
	\begin{align*}
	\delta_M = \delta  P_M,
	\end{align*}
	Since  $\|\delta\|_{L^\infty(\mathbb  R^2)}\le C/\rho_K^2$ and $\|\nabla \delta\|_{L^\infty(\mathbb  R^2)}\le C/\rho_K^3$,  by \eqref{inverse} and the assumption that the triangle mesh is quasi-uniform,  we have 
	\begin{align*}
	\|\nabla \delta_M\|_{L^t(B)} & = 	\|P_M\nabla\delta  + \delta\nabla P_M\|_{L^t(B)}\\
	& \le \|\nabla\delta\|_{L^\infty(B)} \|P_M\|_{L^t(B)}+\|\delta\|_{L^\infty(B)} \|\nabla P_M\|_{L^t(B)}\\
	&\le Ch^{2/t-3}.
	\end{align*}
\end{proof}

The main idea behind the proof of $L^\infty$ norm estimates is to use the so called smooth $\delta_M$ function, which was described in \Cref{delta_function_1}. Given a scalar function $\delta_1$ and a vector $\bm \delta_2$ of the above type,  we define two regularized Green's functions for problem \eqref{Poisson} in mixed form:
\begin{equation}\label{Green1}
\begin{split}
c\bm \Phi_1+\nabla \Psi_1&= 0\qquad\qquad\text{in}\ \Omega,\\
\nabla\cdot\bm  \Phi_1&= \delta_1\qquad\quad~~\text{in}\ \Omega,\\
\Psi_1 &= 0\qquad\qquad\text{on}\ \partial\Omega,
\end{split}
\end{equation}
and
\begin{equation}\label{Green2}
\begin{split}
c\bm \Phi_2+\nabla \Psi_2&= \bm \delta_2\qquad\qquad\text{in}\ \Omega,\\
\nabla\cdot\bm  \Phi_2&= 0\qquad\qquad~\text{in}\ \Omega,\\
\Psi_2 &= 0\qquad\qquad~\text{on}\ \partial\Omega.
\end{split}
\end{equation}

{We need two auxiliary results before starting the proof of  \Cref{main_result_Linfty_norm}.  The first concerns bounds on the regularized Green's function $(\Psi_1, \Psi_2)$:}
\begin{lemma}\label{regularity_Green_function}
	Let $\Psi_1$ and $\Psi_2$ be the solution of \eqref{Green1} and \eqref{Green2}, respectively. If assumption (A) holds, then we have:
	\begin{subequations}
		\begin{gather}
		\| D^2 \Psi_1\|_{L^2(\Omega)} \le Ch^{-1}, \;  \|\sigma D^2 \Psi_1\|_{L^2(\Omega)} \le C|\log h|^{1/2}, \; \|D^2\Psi_1\|_{L^1(\Omega)} \le C|\log h|,\label{regu_delta_1}\\
		\|D^2 \Psi_2\|_{L^2(\Omega)} \le Ch^{-2}, \;  \|\sigma D^2 \Psi_2\|_{L^2(\Omega)} \le Ch^{-1}, \; \|D^2\Psi_2\|_{L^1(\Omega)} \le Ch^{-1}|\log h|^{1/2}.\label{regu_delta_2}
		\end{gather}
	\end{subequations}
\end{lemma}
\begin{proof}
	The proof of \eqref{regu_delta_1} can be found in \cite[equation (3.6)]{wang1989asymptotic}. By the elliptic regularity of Poisson problem and \eqref{delta_function_5} we have 
	\begin{align*}
	\|D^2 \Psi_2\|_{L^2(\Omega)} \le C\|\nabla\cdot \bm \delta_2\|_{L^2(\Omega)} \le Ch^{-2}.
	\end{align*}
	The remaining estimations in \eqref{regu_delta_2} can be found in \cite[Lemma 3.2]{Duran_M2AN_1988} and \cite[equation (3.12d)]{wang1989asymptotic}.
\end{proof}

Next,  we define the weight $\sigma$ by:
\begin{align} \label{def_sigma} 
\sigma(\bm x) = (|\bm x - \bm x_0|^2 +    {h^2})^{1/2},
\end{align}
where  $\bm x_0$ is a point close to that where the {relevant} maximum is attained (the center of the inscribed circle in the 
triangle containing the maximum).
{In the second auxiliary lemma} we summarize some properties of the function $\sigma$,  which will be use later.
\begin{lemma}[{\cite[Equation (2.13)]{wang1989asymptotic}}]
	For any $\alpha\in \mathbb R$ there is a constant $C$ independent of $\alpha$ such that the function $\sigma$ has the following  properties:
	\begin{subequations}
		\begin{gather}
		\frac{\max_{\bm x\in K} \sigma(\bm x)^\alpha}{\min_{\bm x\in K} \sigma(\bm x)^\alpha}  \le C , \forall K\in \mathcal T_h, \label{pro_sigma_1}\\
		\left|\nabla^k (\sigma(\bm x)^\alpha)\right|\le C \sigma(\bm x)^{\alpha-k},\label{pro_sigma_2}\\
		\int_{\Omega} \sigma (\bm x)^{-2} ~{\rm d}\bm x \le C|\log h|.\label{sigma_lemma_2_constant}
		\end{gather}
	\end{subequations}
	
\end{lemma}

We split the proof of \Cref{main_result_Linfty_norm} into four steps. First, we shall obtain the $L^1$ norm approximations error of the solution  of \eqref{Green1} and  \eqref{Green2}. Second, we prove the  $L^\infty$ norm stability of $\bm q_h$ and $u_h$. Next, we obtain the $L^\infty$ norm error estimates of $\bm q-\bm q_h$ and $u - u_h$. Finally, we obtain the $L^\infty$ norm error estimates of the postprocessed solution $u_h^\star$. 

\paragraph*{Step 1: $L^1$ norm error estimates for  the regularized Green's functions}																																																																														

Let $(\bm \Phi_{1,h}, \Psi_{1,h},\widehat \Psi_{1,h})$ and $(\bm \Phi_{2,h}, \Psi_{2,h}, \widehat \Psi_{2,h})$ be the HDG solution of \eqref{Green1} and \eqref{Green2}, respectively, i.e.,
\begin{subequations}
	\begin{align}
	\mathscr B(\bm \Phi_{1,h}, \Psi_{1,h}, \widehat{\Psi}_{1,h}; \bm v_h, w_h,\widehat w_h) & = (\delta_1, w_h)_{\mathcal T_h},\label{Green_HDG_1}\\
	\mathscr B(\bm \Phi_{2,h}, \Psi_{2,h}, \widehat{\Psi}_{2,h}; \bm v_h, w_h,\widehat w_h) & = ( \bm \delta_2, \bm v_h)_{\mathcal T_h}\label{Green_HDG_2}
	\end{align}
	for all $(\bm v_h, w_h, \widehat w_h)\in\bm{V}_h\times W_h\times \widehat W_h$. The existence and uniqueness of these solutions follow by standard HDG theory \cite{Cockburn_Gopalakrishnan_Lazarov_Unify_SINUM_2009}.
\end{subequations}	

Our goal in this step is to prove the upcoming \Cref{weight_err_thm}. To start we summarize some relevant results in the following:
\begin{subequations}
	\begin{gather}
	\mathscr B(\bm \Pi_V \bm \Phi_1, \Pi_W  \Psi_1, \Pi_k^\partial   \Psi_1; \bm v_h, w_h,\widehat w_h)  = (\bm \Pi_V \bm \Phi_1 - \bm \Phi_1, \bm v_h)_{\mathcal T_h} - (\delta_1, w_h)_{\mathcal T_h},\label{pro_error_equa_Phi}\\
	\mathscr B(\bm \Pi_V \bm \Phi_2, \Pi_W  \Psi_2, \Pi_k^\partial   \Psi_2; \bm v_h, w_h,\widehat w_h)  = (\bm \Pi_V \bm \Phi_2 - \bm \Phi_2, \bm v_h)_{\mathcal T_h}+(\bm \delta_2, \bm v_h)_{\mathcal T_h}, \label{pro_error_equa_Psi}\\
	\mathscr B(\bm \Pi_V \bm \Phi_1 - \bm \Phi_{1,h}, \Pi_W  \Psi_1 -  \Psi_{1,h}, \Pi_k^\partial   \Psi_1 - \widehat  \Psi_{1,h}; \bm v_h, w_h,\widehat w_h)  = (\bm \Pi_V \bm \Phi_1 - \bm \Phi_1, \bm v_h)_{\mathcal T_h},\label{error_equa_Phi}\\
	\mathscr B(\bm \Pi_V \bm \Phi_2 - \bm \Phi_{2,h}, \Pi_W  \Psi_2 -  \Psi_{2,h}, \Pi_k^\partial   \Psi_2 - \widehat  \Psi_{2,h}; \bm v_h, w_h,\widehat w_h)  = (\bm \Pi_V \bm \Phi_2 - \bm \Phi_2, \bm v_h)_{\mathcal T_h},\label{error_equa_Psi}\\
	\|\bm \Pi_V \bm \Phi_1 - \bm \Phi_{1,h}\|_{L^2(\Omega)} \le C, \quad \|\Pi_W\Psi_1 -  \Psi_{1,h}\|_{L^2(\Omega)} \le Ch^{\min\{k,1\}},\label{L2_error_Phi}\\
	\|\bm \Pi_V \bm \Phi_2 - \bm \Phi_{2,h}\|_{L^2(\Omega)} \le Ch^{-1}, \quad \|\Pi_W\Psi_2 -  \Psi_{2,h}\|_{L^2(\Omega)} \le Ch^{\min\{k,1\}-1}.\label{L2_error_Psi}
	\end{gather}
\end{subequations}

The proof of \eqref{pro_error_equa_Phi} and \eqref{error_equa_Phi} can be found in  \cite[Lemma 3.6]{ChenCockburnSinglerZhang1} and the proof of \eqref{pro_error_equa_Psi} and \eqref{error_equa_Psi} is similar. The proof of \eqref{L2_error_Phi} and \eqref{L2_error_Psi} can be found in \cite[Theorem 3.1 and Theorem 4.1]{Cockburn_Gopalakrishnan_Sayas_Porjection_MathComp_2010} and the regularity of of   regularized Green’s functions in \Cref{regularity_Green_function}.

We now present a series of lemmas providing convergence estimates for the projections  used in our analysis.

\begin{lemma}\label{L2_noundness_sigma} 
	For any integer $k\ge 0$, $K\in\mathcal T_h$ and $\alpha\in\mathbb R$, let $\bm \Pi_k^{o} $ be the standard $L^2$ projection (see \eqref{L2_do}), then for $\bm v\in \bm H^{k+1}(K)$ we have 
	\begin{subequations}
		\begin{align}\label{L2_prijection_sigma}
		\left\|\sigma^{\alpha}(\bm v-\bm \Pi_k^{o} \bm v)\right\|_{\bm L^2(K)}  \le Ch_K^{k+1} \|\sigma^{\alpha} \nabla^{k+1} \bm v\|_{\bm L^2(K)}.
		\end{align}
		Furthermore, let $w\in H^{k+1}(K)$, then $(\sigma^2\bm v,\sigma^2 w)$ is in the domain of $\Pi_h$  and we have
		\begin{align}\label{LV_prijection_sigma}
		\left\|\sigma^{\alpha}(\bm v-\bm \Pi_V \bm v)\right\|_{\bm L^2(K)}  \le Ch_K^{k+1}\left( \|\sigma^{\alpha}\nabla^{k+1} \bm v\|_{\bm L^2(K)} +   \|\sigma^{\alpha}\nabla^{k+1}  w\|_{ L^2(K)}\right).
		\end{align}
	\end{subequations}
\end{lemma}
\begin{proof}
	We only prove \eqref{LV_prijection_sigma} because the proof of \eqref{L2_prijection_sigma} is similar.
	\begin{align*}
	\left\|\sigma^{\alpha}(\bm v-\bm \Pi_V \bm v)\right\|_{\bm L^2(K)}&\le \max_{\bm x\in K} \{\sigma^{\alpha}\} ~\|\bm v-\bm \Pi_V \bm v\|_{\bm L^2(K)}\\
	&\le Ch_K^{k+1}\max_{\bm x\in K} \{\sigma^{\alpha}\} \left(|\nabla^{k+1} \bm v|_{\bm L^2(K)} +|\nabla^{k+1} w|_{ L^2(K)} \right)&\textup{by } \eqref{Proerr_q}\\
	&\le Ch_K^{k+1}\min_{\bm x\in K} \{\sigma^{\alpha}\} \left(|\nabla^{k+1} \bm v|_{\bm L^2(K)} +|\nabla^{k+1} w|_{ L^2(K)} \right) &\textup{by } \eqref{pro_sigma_1}\\
	&\le Ch_K^{k+1}\left( \|\sigma^{\alpha}\nabla^{k+1} \bm v\|_{\bm L^2(K)} +   \|\sigma^{\alpha}\nabla^{k+1}  w\|_{ L^2(K)}\right).
	\end{align*}
\end{proof}

\begin{lemma}\label{key_in} 
	Let $(\bm v_h, w_h) \in \bm V_h\times W_h$, then for any integer $k\ge 0$, we have
	\begin{subequations}
		\begin{align}
		\left\|\sigma^{-1}\left(\sigma^{2} \bm v_h-\bm \Pi_k^{o}(\sigma^{2} \bm v_h)\right)\right\|_{\bm L^2(K)} &\le  Ch\left\| \bm v_h\right\|_{\bm L^2(K)},\label{error_sigma_Pi}\\
		\left\|\sigma^{-1}\left(\sigma^{2} \bm v_h-\bm \Pi_V (\sigma^{2} \bm v_h)\right)\right\|_{\bm L^2(K)}&\le  Ch (\left\| \bm v_h\right\|_{\bm L^2(K)} + \left\|w_h\right\|_{\bm L^2(K)}).\label{error_sigma_Pi_V}
		\end{align}
	\end{subequations}
\end{lemma}

\begin{proof} 
	Notice that $\bm v_h|_K\in \mathcal [P^k(K)]^2$, i.e., $\nabla^{k+1} \bm v_h=0$. Then by \Cref{L2_noundness_sigma} we have
	\begin{align*}
	\left\|\sigma^{-1}\left(\sigma^2 \bm v_h-\bm \Pi_k^{o}(\sigma^2 \bm v_h)\right)\right\|_{\bm L^2(K)} & \le Ch^{k+1}\left\|\sigma^{-1} \left(\nabla^{k+1}(\sigma^2 \bm v_h)\right)\right\|_{\bm L^2(K)} \\
	&  =  Ch^{k+1}\left\|\sigma^{-1}\sum_{j=1}^{k+1}\nabla^j(\sigma^2)\nabla^{k+1-j}\bm v_h\right\|_{\bm L^2(K)}\\
	&\le Ch\|\bm v_h\|_{\bm L^2(K)}.
	\end{align*}
	where we  applied \eqref{pro_sigma_2} and \Cref{Inverse_inequality} to the above inequality. This proves \eqref{error_sigma_Pi} and  the proof of \eqref{error_sigma_Pi_V} is the same, hence we omit the details here.
\end{proof}

\begin{lemma}\label{weight_err_thm}
	Let $(\bm \Phi_1,\Psi_1)$ and  $(\bm \Phi_{1,h}, \Psi_{1,h},\widehat  \Psi_{1,h})$ be the solution of \eqref{Green1} and \eqref{Green_HDG_1} respectively, and $(\bm \Phi_2,\Psi_2)$ and  $(\bm \Phi_{2,h},\Psi_{2,h}, \widehat \Psi_{2,h})$ be the solutions of \eqref{Green2} and \eqref{Green_HDG_2}. If  assumption (A) holds and $k\geqslant 1$, then we have:
	\begin{subequations}\label{weight_err}
		\begin{align}
		\|\sigma (c\bm \Phi_1 - c\bm \Phi_{1,h})\|_{\bm L^2(\Omega)}   &\le  Ch(|\log h|^{1/2}+1),\label{weight_err_1}\\
		\|\sigma (c\bm \Phi_2 - c\bm \Phi_{2,h})\|_{\bm L^2(\Omega)} &\le  C.\label{weight_err_2}
		\end{align}	
	\end{subequations}
\end{lemma}
\begin{proof}
	Let $	\mathcal E_h^{\bm \Phi_1} = \bm \Pi_V \bm \Phi_1 - \bm \Phi_{1,h}, 		\mathcal E_h^{ \Psi_1} =  \Pi_W  \Psi_1 -  \Psi_{1,h}, 
	\mathcal E_h^{\widehat  \Psi_1} =  \Pi_k^\partial   \Psi_1 - \widehat  \Psi_{1,h}$.
	On the one hand, by the definition of $\mathscr B$ in \eqref{def_B} we obtain:
	\begin{align*}
	\hspace{1em}&\hspace{-1em}\mathscr B(\mathcal E_h^{\bm \Phi_1}, \mathcal E_h^{\Psi_1}, \mathcal E_h^{\widehat \Psi_1}; \sigma^2\mathcal E_h^{\bm \Phi_1}, -\sigma^2\mathcal E_h^{\Psi_1}, -\sigma^2\mathcal E_h^{\widehat \Psi_1})\\
	&=(c\mathcal E_h^{\bm  \Phi_1},  \sigma^2\mathcal E_h^{\bm  \Phi_1})_{{\mathcal{T}_h}}- ( \mathcal E_h^{\Psi_1}, \nabla\cdot(\sigma^2\mathcal E_h^{\bm \Phi_1}))_{{\mathcal{T}_h}}+\langle \mathcal E_h^{\widehat \Psi_1}, \sigma^2\mathcal E_h^{\bm \Phi_1}\cdot \bm{n} \rangle_{\partial{{\mathcal{T}_h}}}\\
	&\quad+ (\nabla\cdot\mathcal E_h^{\bm  \Phi_1},   \sigma^2\mathcal E_h^{\Psi_1})_{{\mathcal{T}_h}}+\langle  \tau (\mathcal E_h^{\Psi_1}- \mathcal E_h^{\widehat \Psi_1}),   \sigma^2\mathcal E_h^{\Psi_1}-\sigma^2\mathcal E_h^{\widehat \Psi_1} \rangle_{\partial{{\mathcal{T}_h}}} - \langle \mathcal E_h^{\bm \Phi} \cdot\bm n,\sigma^2\mathcal E_h^{\widehat \Psi_1} \rangle_{\partial \mathcal T_h}\\
	&=(c\mathcal E_h^{\bm \Phi_1},  \sigma^2\mathcal E_h^{\bm  \Phi_1})_{{\mathcal{T}_h}}- ( \mathcal E_h^{\Psi_1}, \sigma^2 \nabla\cdot \mathcal E_h^{\bm  \Phi_1})_{{\mathcal{T}_h}} - ( \mathcal E_h^{\Psi_1}, 2\sigma \nabla\sigma\mathcal E_h^{\bm  \Phi_1})_{{\mathcal{T}_h}} +\langle \mathcal E_h^{\widehat \Psi_1}, \sigma^2\mathcal E_h^{\bm \Phi_1}\cdot \bm{n} \rangle_{\partial{{\mathcal{T}_h}}}\\
	&\quad+ (\nabla\cdot\mathcal E_h^{\bm  \Phi_1},   \sigma^2\mathcal E_h^{\Psi_1})_{{\mathcal{T}_h}}+\langle  \tau (\mathcal E_h^{\Psi_1}- \mathcal E_h^{\widehat \Psi_1}),   \sigma^2\mathcal E_h^{\Psi_1}-\sigma^2\mathcal E_h^{\widehat \Psi_1} \rangle_{\partial{{\mathcal{T}_h}}} - \langle \mathcal E_h^{\bm  \Phi_1} \cdot\bm n,\sigma^2\mathcal E_h^{\widehat \Psi_1} \rangle_{\partial \mathcal T_h}.
	\end{align*}
	This gives
	\begin{align}\label{equation_1}
	\begin{split}
	\hspace{1em}&\hspace{-1em} \mathscr B(\mathcal E_h^{\bm \Phi_1}, \mathcal E_h^{\Psi_1}, \mathcal E_h^{\widehat \Psi_1}; \sigma^2\mathcal E_h^{\bm \Phi_1}, -\sigma^2\mathcal E_h^{\Psi_1}, -\sigma^2\mathcal E_h^{\widehat \Psi_1})\\
	& =  (c\mathcal E_h^{\bm\Phi_1},   \sigma^2\mathcal E_h^{\bm\Phi_1})_{{\mathcal{T}_h}} + \|\sigma \sqrt \tau (\mathcal E_h^{\Psi_1}- \mathcal E_h^{\widehat \Psi_1})\|_{L^2(\partial\mathcal T_h)}^2 -  ( \mathcal E_h^{\Psi_1}, 2\sigma \nabla\sigma\cdot\mathcal E_h^{\bm  \Phi_1})_{{\mathcal{T}_h}}.
	\end{split}
	\end{align}
	
	We use  definition \eqref{HDG_projection_operator} with $\bm q=\sigma^2\mathcal E_h^{\bm\Phi_1}, u=-\sigma^2\mathcal E_h^{\psi_1}$.
	On the other hand, by the error equation \eqref{error_equa_Phi} are get 
	\begin{align*}
	\hspace{0.1em}&\hspace{-0.1em}\mathscr B(\mathcal E_h^{\bm \Phi_1}, \mathcal E_h^{\Psi_1}, \mathcal E_h^{\widehat \Psi_1}; \sigma^2\mathcal E_h^{\bm \Phi_1}, -\sigma^2\mathcal E_h^{\Psi_1}, -\sigma^2\mathcal E_h^{\widehat \Psi_1})\\
	& = \mathscr B(\mathcal E_h^{\bm \Phi_1}, \mathcal E_h^{\Psi_1}, \mathcal E_h^{\widehat \Psi_1}; \sigma^2\mathcal E_h^{\bm  \Phi_1} - \bm \Pi_V (\sigma^2\mathcal E_h^{\bm \Phi_1}), { -\sigma^2\mathcal E_h^{\Psi_1}- \Pi_W(-\sigma^2\mathcal E_h^{\Psi_1})}, -(\sigma^2\mathcal E_h^{\widehat \Psi_1} - \Pi_k^\partial(\sigma^2\mathcal E_h^{\widehat \Psi_1})))\\
	&\quad  + \mathscr B(\mathcal E_h^{\bm  \Phi_1}, \mathcal E_h^{\Psi_1}, \mathcal E_h^{\widehat \Psi_1}; \bm \Pi_V (\sigma^2\mathcal E_h^{\bm \Phi_1}), {  \Pi_W(-\sigma^2\mathcal E_h^{\Psi_1})},- \Pi_k^\partial(\sigma^2\mathcal E_h^{\widehat \Psi_1}))\\
	& = \mathscr B(\mathcal E_h^{\bm  \Phi_1}, \mathcal E_h^{\Psi_1}, \mathcal E_h^{\widehat \Psi_1}; \sigma^2\mathcal E_h^{\bm  \Phi_1} - \bm \Pi_V(\sigma^2\mathcal E_h^{\bm \Phi_{1}}), { -\sigma^2\mathcal E_h^{\Psi_1}- \Pi_W(-\sigma^2\mathcal E_h^{\Psi_1})}, -(\sigma^2\mathcal E_h^{\widehat \Psi_1} - \Pi_k^\partial(\sigma^2\mathcal E_h^{\widehat \Psi_1})))\\
	&\quad  + (\bm \Pi_V \bm  \Phi_1- \bm  \Phi_1, \bm \Pi_V(\sigma^2\mathcal E_h^{\bm \Phi_1}))_{\mathcal T_h}.
	\end{align*}
	Next, we use the definition of $\mathscr B$ in \eqref{def_B} again to get:
	\begin{align*}
	\hspace{1em}&\hspace{-1em} \mathscr B(\mathcal E_h^{\bm  \Phi_1}, \mathcal E_h^{\Psi_1}, \mathcal E_h^{\widehat \Psi_1}; \sigma^2\mathcal E_h^{\bm  \Phi_1} - \bm \Pi_V(\sigma^2\mathcal E_h^{\bm \Phi_{1}}), { -\sigma^2\mathcal E_h^{\Psi_1}- \Pi_W(-\sigma^2\mathcal E_h^{\Psi_1})}, -(\sigma^2\mathcal E_h^{\widehat \Psi_1} - \Pi_k^\partial(\sigma^2\mathcal E_h^{\widehat \Psi_1})))\\
	&=(\mathcal E_h^{\bm  \Phi_1}, \sigma^2\mathcal E_h^{\bm  \Phi_1} - \bm \Pi_V(\sigma^2\mathcal E_h^{\bm \Phi_1}))_{{\mathcal{T}_h}}- ( \mathcal E_h^{\Psi_1}, \nabla\cdot(\sigma^2\mathcal E_h^{\bm  \Phi_1} - \bm \Pi_V(\sigma^2\mathcal E_h^{\bm  \Phi_1})))_{{\mathcal{T}_h}}\\
	&\quad +\langle \mathcal E_h^{\widehat \Psi_1}, (\sigma^2\mathcal E_h^{\bm  \Phi_1} - \bm \Pi_V(\sigma^2\mathcal E_h^{\bm  \Phi_1}){ )}\cdot \bm{n} \rangle_{\partial{{\mathcal{T}_h}}}+ (\nabla\cdot\mathcal E_h^{\bm  \Phi_1},   { \sigma^2\mathcal E_h^{\Psi_1}+ \Pi_W(-\sigma^2\mathcal E_h^{\Psi_1})})_{{\mathcal{T}_h}}\\
	&\quad +\langle  \tau (\mathcal E_h^{\Psi_1}- \mathcal E_h^{\widehat \Psi_1}),   ({ \sigma^2\mathcal E_h^{\Psi_1}+ \Pi_W(-\sigma^2\mathcal E_h^{\Psi_1})}) -(\sigma^2\mathcal E_h^{\widehat \Psi_1} - \Pi_k^\partial(\sigma^2\mathcal E_h^{\widehat \Psi_1}))  \rangle_{\partial{{\mathcal{T}_h}}} \\
	&\quad - \langle \mathcal E_h^{\bm  \Phi_1} \cdot\bm n,\sigma^2\mathcal E_h^{\widehat \Psi_1} - \Pi_k^\partial(\sigma^2\mathcal E_h^{\widehat \Psi_1}) \rangle_{\partial \mathcal T_h}+ (\bm \Pi_V \bm  \Phi_1- \bm  \Phi_1, \bm \Pi_V(\sigma^2\mathcal E_h^{\bm  \Phi_1}))\\
	&=(\mathcal E_h^{\bm  \Phi_1}, \sigma^2\mathcal E_h^{\bm  \Phi_1} - \bm \Pi_V(\sigma^2\mathcal E_h^{\bm \Phi_1}))_{{\mathcal{T}_h}} + (\nabla \mathcal E_h^{\Psi_1}, \sigma^2\mathcal E_h^{\bm  \Phi_1} - \bm \Pi_V(\sigma^2\mathcal E_h^{\bm  \Phi_1}))_{{\mathcal{T}_h}}\\
	&\quad -\langle \mathcal E_h^{ \Psi_1} -  \mathcal E_h^{\widehat \Psi_1}, (\sigma^2\mathcal E_h^{\bm  \Phi_1} - \bm \Pi_V(\sigma^2\mathcal E_h^{\bm  \Phi_1}){ )}\cdot \bm{n} \rangle_{\partial{{\mathcal{T}_h}}}+ (\nabla\cdot\mathcal E_h^{\bm  \Phi_1},   { \sigma^2\mathcal E_h^{\Psi_1}+ \Pi_W(-\sigma^2\mathcal E_h^{\Psi_1})})_{{\mathcal{T}_h}}\\
	&\quad +\langle  \tau (\mathcal E_h^{\Psi_1}- \mathcal E_h^{\widehat \Psi_1}),  { \sigma^2\mathcal E_h^{\Psi_1}+ \Pi_W(-\sigma^2\mathcal E_h^{\Psi_1})}\rangle_{\partial{{\mathcal{T}_h}}} + (\bm \Pi_V \bm  \Phi_1- \bm  \Phi_1, \bm \Pi_V(\sigma^2\mathcal E_h^{\bm  \Phi_1}))_{\mathcal T_h},
	\end{align*}
	where we used integration by parts in the above equation. Notice that $(\sigma^2\mathcal E_h^{\bm  \Phi_1}, -\sigma^2\mathcal E_h^{\Psi_1})$ is in the domain of $\Pi_h$ (see \eqref{HDG_projection_operator}), then by \eqref{projection_operator_3} and the fact that $(\mathcal E_h^{ \Psi_1} -  \mathcal E_h^{\widehat \Psi_1})|_F \in  \mathcal P^k(F)$ for all $F\in \mathcal F_h$, we have 
	\begin{align*} 
	\langle \mathcal E_h^{ \Psi_1} -  \mathcal E_h^{\widehat \Psi_1}, ( \bm \Pi_V(\sigma^2\mathcal E_h^{\bm  \Phi_1})-\sigma^2\mathcal E_h^{\bm  \Phi_1})\cdot \bm{n} \rangle_{\partial{{\mathcal{T}_h}}}+\langle  \tau (\mathcal E_h^{\Psi_1}- \mathcal E_h^{\widehat \Psi_1}),  { \sigma^2\mathcal E_h^{\Psi_1}+ \Pi_W(-\sigma^2\mathcal E_h^{\Psi_1})}\rangle_{\partial{{\mathcal{T}_h}}} =0.
	\end{align*}
	Furthermore, by \eqref{projection_operator_1}-\eqref{projection_operator_2} and the fact that $\nabla \mathcal E_h^{\Psi_1}|_K\in [\mathcal P^{k-1}(K)]^2$ and $\nabla\cdot\mathcal E_h^{\bm  \Phi_1}|_K \in \mathcal P^{k-1}(K)$, then we have 
	\begin{align*}
	(\nabla \mathcal E_h^{\Psi_1},\sigma^2\mathcal E_h^{\bm  \Phi_1} - \bm \Pi_V(\sigma^2\mathcal E_h^{\bm \Phi_1}) )_{{\mathcal{T}_h}} &= 0,\\
	(\nabla\cdot\mathcal E_h^{\bm  \Phi_1},  { \sigma^2\mathcal E_h^{\Psi_1}+ \Pi_W(-\sigma^2\mathcal E_h^{\Psi_1})})_{{\mathcal{T}_h}}& = 0.
	\end{align*}
	This gives
	\begin{align}\label{equation_2}
	\begin{split}
	\hspace{1em}&\hspace{-1em}\mathscr B(\mathcal E_h^{\bm \Phi_1}, \mathcal E_h^{\Psi_1}, \mathcal E_h^{\widehat \Psi_1}; \sigma^2\mathcal E_h^{\bm \Phi_1}, -\sigma^2\mathcal E_h^{\Psi_1}, -\sigma^2\mathcal E_h^{\widehat \Psi_1})\\
	&=(\mathcal E_h^{\bm  \Phi_1}, \sigma^2\mathcal E_h^{\bm  \Phi_1} - \bm \Pi_V(\sigma^2\mathcal E_h^{\bm  \Phi_1}))_{{\mathcal{T}_h}}+ (\bm \Pi_V \bm  \Phi_1- \bm  \Phi_1, \bm \Pi_V(\sigma^2\mathcal E_h^{\bm  \Phi_1})).
	\end{split}
	\end{align}
	Comparing with \eqref{equation_1} and \eqref{equation_2} we have
	\begin{align*}
	\hspace{1em}&\hspace{-1em} (c\mathcal E_h^{\bm\Phi_1},   \sigma^2\mathcal E_h^{\bm\Phi_1})_{{\mathcal{T}_h}} + \|\sigma \sqrt \tau (\mathcal E_h^{\Psi_1}- \mathcal E_h^{\widehat \Psi_1})\|_{L^2(\partial\mathcal T_h)}^2 \\
	&=   ( \mathcal E_h^{\Psi_1}, 2\sigma \nabla\sigma\cdot \mathcal E_h^{\bm \Phi_1})_{{\mathcal{T}_h}} + (\mathcal E_h^{\bm \Phi_1}, \sigma^2\mathcal E_h^{\bm \Phi_1} - \bm \Pi_V(\sigma^2\mathcal E_h^{\bm \Phi_1}))_{{\mathcal{T}_h}}\\
	&\quad + (\bm \Pi_V \bm \Phi_1- \bm \Phi_1, \bm \Pi_V(\sigma^2\mathcal E_h^{\bm \Phi_1}) -\sigma^2\mathcal E_h^{\bm \Phi_1} )_{{\mathcal{T}_h}} + (\bm \Pi_V \bm \Phi_1- \bm \Phi_1, \sigma^2\mathcal E_h^{\bm \Phi_1} )_{{\mathcal{T}_h}}\\
	& = I_1  + I_2 + I_3 + I_4.
	\end{align*}
	For the first term $I_1$, we use \eqref{pro_sigma_2}, Young's inequality,  \eqref{L2_error_Phi} and $k\geq 1$ to get
	\begin{align*}
	|I_1| \le \frac 1 4 (c\mathcal E_h^{\bm\Phi_1},   \sigma^2\mathcal E_h^{\bm\Phi_1})_{{\mathcal{T}_h}}  + C \|\mathcal E_h^{\Psi_1}\|^2_{ L^2(\Omega)} \le  \frac 1 4  \|\sigma \mathcal E_h^{\bm \Phi_1}\|_{\bm L^2(\Omega)}^2  + C h^2.
	\end{align*}
	For the second term $I_2$, we use Young's inequality, \eqref{error_sigma_Pi_V} and  \eqref{L2_error_Phi} to get
	\begin{align*}
	|I_2| &=\left|(\sigma\mathcal E_h^{\bm \Phi_1}, \sigma^{-1}(\sigma^2\mathcal E_h^{\bm \Phi_1} - \bm \Pi_V(\sigma^2\mathcal E_h^{\bm \Phi_1})))_{{\mathcal{T}_h}}\right|\\
	& \le \frac 1 4 (c\mathcal E_h^{\bm\Phi_1},   \sigma^2\mathcal E_h^{\bm\Phi_1})_{{\mathcal{T}_h}} + C \left\|\sigma^{-1}(\sigma^2\mathcal E_h^{\bm \Phi_1} - \bm \Pi_V(\sigma^2\mathcal E_h^{\bm \Phi_1})))_{{\mathcal{T}_h}}\right\|_{\bm L^2(\Omega)}^2\\
	&\le \frac 1 4 (c\mathcal E_h^{\bm\Phi_1},   \sigma^2\mathcal E_h^{\bm\Phi_1})_{{\mathcal{T}_h}}  + Ch^2(\| \mathcal E_h^{\bm \Phi_1}\|_{\bm L^2(\Omega)}^2   + \| \mathcal E_h^{ \Psi_1}\|_{ L^2(\Omega)}^2)\\
	&\le \frac 1 4 (c\mathcal E_h^{\bm\Phi_1},   \sigma^2\mathcal E_h^{\bm\Phi_1})_{{\mathcal{T}_h}} + Ch^2.
	\end{align*}
	For the third  term $I_3$, we use Young's inequality, \eqref{error_sigma_Pi_V},  \eqref{L2_error_Phi}, \eqref{LV_prijection_sigma} and \eqref{regu_delta_1} to get
	\begin{align*}
	|I_3| &=\left| (\sigma(\bm \Pi_V \bm \Phi_1- \bm \Phi_1), \sigma^{-1}(\bm \Pi_V(\sigma^2\mathcal E_h^{\bm \Phi_1}) -\sigma^2\mathcal E_h^{\bm \Phi_1} ))_{{\mathcal{T}_h}} \right|\\
	& \le \frac 1 2  \|\sigma(\bm \Pi_V \bm \Phi_1- \bm \Phi_1)\|_{\bm L^2(\Omega)}^2  + \frac 1 2 \left\|\sigma^{-1}(\sigma^2\mathcal E_h^{\bm \Phi_1} - \bm \Pi_V(\sigma^2\mathcal E_h^{\bm \Phi_1})))_{{\mathcal{T}_h}}\right\|_{\bm L^2(\Omega)}^2\\
	&\le C h^2  \|\sigma D^2\Psi_1 \|_{ L^2(\Omega)}^2  + Ch^2(\| \mathcal E_h^{\bm \Phi_1}\|_{\bm L^2(\Omega)}^2   + \| \mathcal E_h^{ \Psi_1}\|_{ L^2(\Omega)}^2)\\
	&\le C h^2( 1+ |\log h|).
	\end{align*}
	For the last term $I_4$, we use Young's inequality,  \eqref{LV_prijection_sigma} and \eqref{regu_delta_1}  to get
	\begin{align*}
	|I_4| &=\left| (\sigma(\bm \Pi_V \bm \Phi_1- \bm \Phi_1), \sigma\mathcal E_h^{\bm \Phi_1} )_{{\mathcal{T}_h}} \right|\\
	& \le \frac 1 4  (c\mathcal E_h^{\bm\Phi_1},   \sigma^2\mathcal E_h^{\bm\Phi_1})_{{\mathcal{T}_h}}  + C \|\sigma(\bm \Pi_V \bm \Phi_1- \bm \Phi_1)\|_{\bm L^2(\Omega)}^2 \\
	&\le  \frac 1 4  (c\mathcal E_h^{\bm\Phi_1},   \sigma^2\mathcal E_h^{\bm\Phi_1})_{{\mathcal{T}_h}}  + C h^2  \|\sigma D^2\Psi \|_{ L^2(\Omega)}^2  \\
	&\le  \frac 1 4  (c\mathcal E_h^{\bm\Phi_1},   \sigma^2\mathcal E_h^{\bm\Phi_1})_{{\mathcal{T}_h}}  + C h^2( 1+ |\log h|).
	\end{align*}
	
\end{proof}

\begin{lemma}\label{dual_L11}
	Let $(\bm \Phi_1,\Psi_1)$ and  $(\bm \Phi_{1,h}, \Psi_{1,h},\widehat \Psi_{1,h})$ be the solutions of \eqref{Green1} and \eqref{Green_HDG_1} respectively, $(\bm \Phi_2,\Psi_2)$ and  let $(\bm \Phi_{2,h}, \Psi_{2,h},\widehat \Psi_{2,h})$ be the solution of \eqref{Green2} and \eqref{Green_HDG_2}.  If assumption (A) holds and $k\geqslant 1$, then we have:
	\begin{subequations}\label{imp_est}
		\begin{align}
		\|\bm \Phi_{1,h}- \bm \Phi_1\|_{\bm L^1(\Omega)} + \|c\bm \Phi_1 - \bm \Pi_{k-1}(c\bm \Phi_1)\|_{\bm L^1(\Omega)} &\le  Ch(|\log h|+1),\label{imp_est_1}\\
		\|\bm \Phi_{2,h}- \bm \Phi_2\|_{\bm L^1(\Omega)} + \|c\bm \Phi_2 - \bm \Pi_{k-1}(c\bm \Phi_2)\|_{\bm L^1(\Omega)} &\le  C(|\log h|^{1/2}+1).\label{imp_est_2}
		\end{align}	
	\end{subequations}
\end{lemma}

\begin{proof}
	By the Cauchy-Schwarz  inequality  and \eqref{weight_err_1} we have 
	\begin{align*}
	\| \bm \Phi_1 -  \bm \Phi_{1,h}\|_{\bm L^1(\Omega)} =   \int_{\Omega} \sigma^{-1} (\sigma| \bm \Phi_1 - \bm \Phi_{1,h}|)~{\rm d}\bm x\le  Ch(1+|\log h|).
	\end{align*}
	Next, by  \eqref{Pro_jec_1} we have 
	\begin{align}\label{L1_2}
	\|c\bm \Phi_1 - \bm \Pi_{k-1}(c\bm \Phi_1)\|_{\bm L^1(\Omega)} \le Ch\|\nabla^2\Psi_1\|_{L^1(\Omega)}\le Ch|\log h|.
	\end{align}
	Then, \eqref{imp_est_1} follows. The proof of \eqref{imp_est_2} is similar to the proof of  \eqref{imp_est_1}.
\end{proof}

\paragraph*{Step 2: Proof of  \eqref{maxmimu_norm_u}-\eqref{maxmimu_norm_q}  in \Cref{main_result_Linfty_norm}}																					
\begin{proof}
	We only prove  \eqref{maxmimu_norm_u} since the proof of \eqref{maxmimu_norm_q} is similar. We choose $\delta_1$ so that $\|u_h\|_{ L^\infty(\Omega)} = (\delta_1, u_h)_{\mathcal T_h}$, then 
	\begin{align*}
	-(\delta_1, u_h) & = \mathscr B(\bm \Phi_{1,h}, \Psi_{1,h}, \widehat{\Psi}_{1,h}; \bm q_h, u_h,\widehat u_h)&\textup{by }\eqref{Green_HDG_1}\\
	& =  \mathscr B( \bm q_h, u_h,\widehat u_h; \bm \Phi_{1,h}, \Psi_{1,h}, \widehat{\Psi}_{1,h})&\textup{by }\eqref{commute}\\
	& =  \mathscr B( \bm q, u, u; \bm \Phi_{1,h}, \Psi_{1,h}, \widehat{\Psi}_{1,h})&\textup{by }\eqref{HDG_exact}\\
	& = \mathscr B(\bm \Phi_{1,h}, \Psi_{1,h}, \widehat{\Psi}_{1,h}; \bm q, u, u)&\textup{by }\eqref{commute}\\
	&  =\mathscr B(\bm \Phi_{1,h} - \bm \Phi_{1}, \Psi_{1,h}-\Psi_{1}, \widehat{\Psi}_{1,h} - {\Psi}_{1}; \bm q, u, u) + \mathscr B(\bm \Phi_{1}, \Psi_{1}, {\Psi}_{1}; \bm q, u, u)\\
	&  =\mathscr B(\bm \Phi_{1,h} - \bm \Phi_{1}, \Psi_{1,h}-\Psi_{1}, \widehat{\Psi}_{1,h} - {\Psi}_{1}; \bm q, u, u) - (\delta_1, u)&\textup{by }\eqref{Green_HDG_1}.
	\end{align*}
	By the definition of $ \mathscr B$ in  \eqref{def_B} we have 
	\begin{align*}
	\hspace{1em}&\hspace{-1em} \mathscr B(\bm \Phi_{1,h} - \bm \Phi_{1}, \Psi_{1,h}-\Psi_{1}, \widehat{\Psi}_{1,h} - {\Psi}_{1}; \bm q, u, u)\\
	& = (c(\bm \Phi_{1,h} - \bm \Phi_{1}), \bm q)_{\mathcal T_h}  - ( \Psi_{1,h}-\Psi_{1},\nabla\cdot \bm q)_{\mathcal T_h}  + \langle\widehat{\Psi}_{1,h} - {\Psi}_{1},\bm q\cdot\bm n \rangle_{\partial\mathcal T_h}\\
	&\quad  -(\nabla\cdot (\bm \Phi_{1,h} - \bm \Phi_{1}), u)_{\mathcal T_h} +  \langle(\bm \Phi_{1,h} - \bm \Phi_{1})\cdot \bm n, u\rangle_{\partial\mathcal T_h}\\
	&  = - ( \Psi_{1,h}-\Psi_{1},\nabla\cdot \bm q)_{\mathcal T_h} =  - ( \Psi_{1,h}-\Psi_{1}, f)_{\mathcal T_h},
	\end{align*}
	where we used integration by parts  and the fact that $ \langle\widehat{\Psi}_{1,h},\bm q\cdot\bm n \rangle_{\partial\mathcal T_h} = 0$ and  $ \langle {\Psi}_{1},\bm q\cdot\bm n \rangle_{\partial\mathcal T_h} = 0$ in the last equality. Combining the above two equations,  and using the fact that $\|\delta_1\|_{L^1(\Omega)}=1$  
	together with the estimates in \eqref{L2_error_Phi} gives
	\begin{align*}
	\|u_h\|_{L^\infty(\Omega)} \le    \|u\|_{L^\infty(\Omega)}  + ( \Psi_{1,h}-\Psi_{1}, f)_{\mathcal T_h} \le \|u\|_{L^\infty(\Omega)}   + Ch^{\min\{k,1\}}\|f\|_{L^2(\Omega)}.
	\end{align*}
	This completes the proof of \eqref{maxmimu_norm_u}.
\end{proof}

\paragraph*{Step 3: Proof of \eqref{thmL_infty_q}-\eqref{thmL_infty_u}  in \Cref{main_result_Linfty_norm}}		
We choose $\delta_1$ and $\bm \delta_2$ such that $ \|\Pi_W u - u_h\|_{ L^\infty(\Omega)} = (\delta_1, \Pi_W u - u_h)_{\mathcal T_h}$, $\|\bm \Pi_V \bm q - \bm q_h\|_{\bm L^\infty(\Omega)} = (\bm  \delta_2, \bm \Pi_V \bm q - \bm q_h) _{\mathcal T_h}$.

\begin{lemma}\label{main_result_for_L_infty}
	Let $(\bm \Phi_{1,h}, \Psi_{1,h},\widehat \Psi_{1,h})$ and $(\bm \Phi_{2,h}, \Psi_{2,h},\widehat \Psi_{2,h})$ be the HDG solutions of \eqref{Green_HDG_1} and \eqref{Green_HDG_2}, respectively. Then we have 
	\begin{align*}
	- (\delta_1, \Pi_W u - u_h)_{\mathcal T_h}   =(c(\bm \Phi_{1,h}-\bm \Phi_1), \bm\Pi_V \bm q - \bm q)_{{\mathcal{T}_h}}   +  (c\bm \Phi_1 - \bm \Pi_{k-1} (c\bm \Phi_1), \bm\Pi_V \bm q - \bm q)_{{\mathcal{T}_h}},\\
	(\bm \delta_2, \bm\Pi_V \bm q - \bm q_h)_{\mathcal T_h}    =(c(\bm \Phi_{2,h}-\bm \Phi_2), \bm\Pi_V \bm q - \bm q)_{{\mathcal{T}_h}}   +  (c\bm \Phi_2 - \bm \Pi_{k-1} (c\bm \Phi_2), \bm\Pi_V \bm q - \bm q)_{{\mathcal{T}_h}}.
	\end{align*}
	
\end{lemma}

\begin{proof}
	We take $(\bm v_h, w_h, \widehat w_h) = (\bm\Pi_V \bm q - \bm q_h, \Pi_W u - u_h, \Pi_k^\partial  u-\widehat u_h)$ in \eqref{Green_HDG_1} to get
	\begin{align*}
	- (\delta_1, \Pi_W u - u_h)_{\mathcal T_h} 
	& =  \mathscr B(\bm \Phi_h,\Psi_h,\widehat \Psi_h; \bm\Pi_V \bm q - \bm q_h, \Pi_W u - u_h, \Pi_k^\partial  u-\widehat u_h)\\
	& =  \mathscr B( \bm\Pi_V \bm q - \bm q_h, \Pi_W u - u_h, \Pi_k^\partial  u-\widehat u_h; \bm \Phi_{1,h},\Psi_{1,h},\widehat \Psi_{1,h}) & \textup{by } \eqref{commute}\\
	& =  \mathscr B( \bm\Pi_V \bm q - \bm q, \Pi_W u - u, \Pi_k^\partial  u- u; \bm \Phi_{1,h},\Psi_{1,h},\widehat \Psi_{1,h})& \textup{by } \eqref{HDG_exact} \\
	& =  \mathscr B(\bm \Phi_{1,h},\Psi_{1,h},\widehat \Psi_{1,h}; \bm\Pi_V \bm q - \bm q, \Pi_W u - u, \Pi_k^\partial  u- u ) & \textup{by } \eqref{commute}\\
	& =(c\bm \Phi_{1,h}, \bm\Pi_V \bm q - \bm q)_{{\mathcal{T}_h}}- (\Psi_{1,h}, \nabla\cdot (\bm\Pi_V \bm q - \bm q))_{{\mathcal{T}_h}}\\
	&\quad +\langle \widehat \Psi_{1,h}, (\bm\Pi_V \bm q - \bm q)\cdot \bm{n} \rangle_{\partial{{\mathcal{T}_h}}}- (\nabla\cdot\bm \Phi_{1,h},   \Pi_W u - u)_{{\mathcal{T}_h}}\\
	&\quad-\langle  \tau ( \Psi_{1,h} - \widehat \Psi_{1,h}) ,   \Pi_W u - u\rangle_{\partial{{\mathcal{T}_h}}} + \langle \bm \Phi_{1,h}\cdot\bm n, \Pi_k^\partial  u - u  \rangle_{\partial \mathcal T_h},
	\end{align*}
	where we used  the definition of $ \mathscr B$ in the last step. Next, by \eqref{projection_operator_1}-\eqref{projection_operator_3}  we have 
	\begin{align*}
	- (\delta_1, \Pi_W u - u_h)_{\mathcal T_h} & =(c\bm \Phi_{1,h}, \bm\Pi_V \bm q - \bm q)_{{\mathcal{T}_h}}- (\Psi_{1,h}, \nabla\cdot (\bm\Pi_V \bm q - \bm q))_{{\mathcal{T}_h}}\\
	&\quad+\langle \widehat \Psi_{1,h}, (\bm\Pi_V \bm q - \bm q)\cdot \bm{n} \rangle_{\partial{{\mathcal{T}_h}}}-\langle  \tau ( \Psi_{1,h} - \widehat \Psi_{1,h}) ,   \Pi_W u - u\rangle_{\partial{{\mathcal{T}_h}}}\\
	& =(c\bm \Phi_{1,h}, \bm\Pi_V \bm q - \bm q)_{{\mathcal{T}_h}} +  (\nabla\Psi_h, \bm\Pi_V \bm q - \bm q)_{{\mathcal{T}_h}}\\
	&\quad-\langle \Psi_{1,h} - \widehat \Psi_{1,h}, (\bm\Pi_V \bm q - \bm q)\cdot \bm{n} \rangle_{\partial{{\mathcal{T}_h}}}-\langle  \tau ( \Psi_{1,h} - \widehat \Psi_{1,h}) ,   \Pi_W u - u\rangle_{\partial{{\mathcal{T}_h}}} \\
	& = (c\bm \Phi_{1,h}, \bm\Pi_V \bm q - \bm q)_{{\mathcal{T}_h}} \\
	& = (c(\bm \Phi_{1,h}-\bm \Phi), \bm\Pi_V \bm q - \bm q)_{{\mathcal{T}_h}}   +  (c\bm \Phi, \bm\Pi_V \bm q - \bm q)_{{\mathcal{T}_h}} \\
	& = (c(\bm \Phi_{1,h}-\bm \Phi), \bm\Pi_V \bm q - \bm q)_{{\mathcal{T}_h}}   +  (c\bm \Phi - \bm \Pi_{k-1} (c\bm \Phi), \bm\Pi_V \bm q - \bm q)_{{\mathcal{T}_h}}.
	\end{align*}
	This gives the proof of the first identity, we omit the proof of the second identity since it follows along the same lines.
\end{proof}

\begin{lemma}\label{projection_infty_error}
	Let $(\bm q, u)$ and $(\bm q_h, u_h)$ be the solution of \eqref{Poisson} and \eqref{HDG_discrete2}, respectively.  If  assumption (A) holds and $k\geqslant 1$, then we have:
	\begin{subequations}\label{prothmL_infty}
		\begin{align}
		\|\Pi_W u- u_h\|_{L^\infty(\Omega)} &\le Ch^{k+2} (|\log h | +1 )(|\bm q|_{\bm W^{k+1,\infty}(\Omega)} + |u|_{ W^{k+1,\infty}(\Omega)}),\label{prothmL_infty_u}\\
		\|\bm \Pi_V \bm q- \bm q_h\|_{\bm L^\infty(\Omega)} &\le Ch^{k+1}(|\log h |^{1/2} +1 )(|\bm q|_{\bm W^{k+1,\infty}(\Omega)} + |u|_{ W^{k+1,\infty}(\Omega)}).\label{prothmL_infty_q}
		\end{align}
	\end{subequations}
	
\end{lemma}

\begin{proof}
	By \Cref{main_result_for_L_infty,dual_L11} we have 
	\begin{align*}
	\| \Pi_W u - u_h\|_{L^\infty(\Omega)}  &=\left|(c(\bm \Phi_{1,h}-\bm \Phi_1), \bm\Pi_V \bm q - \bm q)_{{\mathcal{T}_h}}   +  (c\bm \Phi_1 - \bm \Pi_{k-1} (c\bm \Phi_1), \bm\Pi_V \bm q - \bm q)_{{\mathcal{T}_h}}\right|\\
	&\le C\|\bm\Pi_V \bm q - \bm q\|_{\bm L^\infty(\Omega)} \left(\|\bm \Phi_{1,h}-\bm \Phi_1\|_{\bm L^1(\Omega)} + \|c\bm \Phi_1 - \bm \Pi_{k-1}(c\bm \Phi_1)\|_{\bm L^1(\Omega)}\right)\\
	& \le Ch^{k+2}(|\log h|+1)(|\bm q|_{\bm W^{k+1,\infty}(\Omega)} + |u|_{ W^{k+1,\infty}(\Omega)}).
	\end{align*}
\end{proof}

As a consequence, a simple application of the triangle inequality and \Cref{projection_infty_error,general_lp_projection_error} gives convergence rates for $\|\bm q -\bm q_h\|_{\bm L^\infty(\Omega)}$ and $\|u - u_h\|_{ L^\infty(\Omega)}$. This completes the proof of \eqref{maxmimu_norm_u}-\eqref{maxmimu_norm_q}  in \Cref{main_result_Linfty_norm}.

\paragraph*{Step 4:  Proof of \eqref{thmL_infty_ustar}  in \Cref{main_result_Linfty_norm}}
\begin{proof}
	First, for all $w_{0}\in \mathcal P^{0}(K)$, we have
	\begin{align}\label{error_pro}
	(\Pi_W u-\Pi_{k+1}^o u, w_{0})_K =(\Pi_W u-u, w_{0})_K + (u - \Pi_{k+1}^o u, w_{0})_K= 0.
	\end{align}
	
	Let $e_h=u_{h}^{\star} - u_h+\Pi_W u-\Pi_{k+1}^o u$,  by \eqref{post_process_1} we obtain
	\begin{align*}
	\|\nabla e_h\|_{L^2(K)}^2&=(\nabla (u_{h}^{\star} - u_{h}),\nabla e_{h} )_K+( \nabla (\Pi_W u -\Pi_{k+1}^o u),\nabla e_h)_K\\
	&=(-\nabla u_{h}- \bm{q}_{h},\nabla e_h )_K+(  \nabla (\Pi_W u-\Pi_{k+1}^o u),\nabla e_h)_K \\
	&=(\nabla (\Pi_W u - u_h)_K-(\bm{q}_h-\bm q)  +\nabla (u - \Pi_{k+1}^o u),\nabla e_h)_K.
	\end{align*}
	Using \Cref{Inverse_inequality} this implies that
	\begin{align}\label{H1}
	\|\nabla e_h\|_{L^2(K)} \le C  (h_K^{-1}\|\Pi_W u- u_h\|_{L^2(K)}+\| \bm{q}_h-\bm{q} \|_{L^2(K)} +\|\nabla(u - \Pi_{k+1} u)\|_{L^2(K)}).
	\end{align}
	By \eqref{post_process_1_b} and \eqref{error_pro} we get $(e_h,1)_K=0$, i.e., $\Pi_0^o e_h = 0$. Then standard estimates for the $L^2$ projection given in  \eqref{H1}  shows that
	\begin{align*}
	\|e_h\|_{L^2(K)} & = \|e_h - \Pi_0^o e_h\|_{L^2(K)} \\
	&\le C h_K \|\nabla e_h\|_{L^2(K)} \\
	&\le C (\|\Pi_W u- u_h\|_{L^2(K)}+h_K\| \bm{q}_h-\bm{q} \|_{L^2(K)} +h_K\|\nabla(u - \Pi_{k+1} u)\|_{L^2(K)}).
	\end{align*}
	Hence, we have
	\begin{align*}
	\|\Pi_{k+1}^o u - u_{h}^{\star}\|_{L^2(K)} &\le C  \|\Pi_W u- u_h\|_{L^2(K)} +Ch\|\bm{q}_{h}-\bm \Pi_V \bm q \|_{L^2(K)}\\
	&\quad  +Ch\|\bm \Pi_V \bm q-\bm{q} \|_{L^2(K)} +Ch\|\nabla(u - \Pi_{k+1}^o u)\|_{L^2(K)} .
	\end{align*}
	
	We use the above inequality and \Cref{Inverse_inequality} to get:
	\begin{align*}
	\|\Pi_{k+1}^o u - u_{h}^{\star}\|_{L^\infty(K)} &\le Ch^{-1}\|\Pi_{k+1}^o u - u_{h}^{\star}\|_{L^2(K)}\\
	& \le   C  h^{-1}\|\Pi_W u- u_h\|_{L^2(K)} +C\|\bm{q}_{h}-\bm \Pi_V \bm q \|_{L^2(K)}\\
	&\quad  +C\|\bm \Pi_V \bm q-\bm{q} \|_{L^2(K)} +C\|\nabla(u - \Pi_{k+1}^o u)\|_{L^2(K)}\\
	& \le   C\|\Pi_W u- u_h\|_{L^\infty(K)} +Ch\|\bm{q}_{h}-\bm \Pi_V \bm q \|_{L^\infty(K)}\\
	&\quad  +C\|\bm \Pi_V \bm q-\bm{q} \|_{L^2(K)} +C\|\nabla(u - \Pi_{k+1}^o u)\|_{L^2(K)}.
	\end{align*}

	Now let $K^\star$ denote the element in which $	\|\Pi_{k+1}^o u - u_{h}^{\star}\|_{L^\infty(\Omega)}  =  \|\Pi_{k+1}^o u - u_{h}^{\star}\|_{L^\infty(K^\star)}$. Then
	\begin{align*}
	\|\Pi_{k+1}^o u - u_{h}^{\star}\|_{L^\infty(\Omega)}  &=  \|\Pi_{k+1}^o u - u_{h}^{\star}\|_{L^\infty(K^\star)}\\
	& \le C\left(\|\Pi_W u- u_h\|_{L^\infty(K^\star)} +h\|\bm{q}_{h}-\bm \Pi_V \bm q \|_{L^\infty(K^\star)}\right)\\
	&\quad  +C \left(\|\bm \Pi_V \bm q-\bm{q} \|_{L^2(K^\star)} +\|\nabla(u - \Pi_{k+1}^o u)\|_{L^2(K^\star)}\right)\\
	&\le C\left(\|\Pi_W u- u_h\|_{L^\infty(\Omega)} +h\|\bm{q}_{h}-\bm \Pi_V \bm q \|_{L^\infty(\Omega)}\right)\\
	&\quad  +C \left(\|\bm \Pi_V \bm q-\bm{q} \|_{L^2(K^\star)} +\|\nabla(u - \Pi_{k+1}^o u)\|_{L^2(K^\star)}\right).
	\end{align*}
	By the estimates in \Cref{approximation_of_L2,Inverse_inequality,general_lp_projection_error,projection_infty_error} and the triangle inequality we get our claimed result.
\end{proof}

\section{Quasi-optimal estimates on interfaces}\label{Quasi_optimal_estimates_on_interfaces}
Let $\Gamma$ be a finite union of line segments such that $\Omega$ is decomposed into finitely many Lipschitz domains by $\Gamma$.  We stress that, while $\Omega$ is assumed to be convex, the subdomains need not be convex. Define $\mathcal F_h^\Gamma$ by
\begin{align*}
\mathcal F_h^\Gamma = \{F\in \mathcal F_h\,:\,\textup{measure} (F\cap \Gamma)>0\}.
\end{align*}

We assume, furthermore, that the triangle mesh $\mathcal T_h$ resolves $\Gamma$. Hence,  $\Gamma$ can be written as the union of $\mathcal O(h^{-1})$ edges in $\mathcal F_h$, i.e., $\bar \Gamma = \bigcup_{F\in\mathcal F_h^\Gamma\subset \mathcal F_h} \bar F$.
\begin{theorem}\label{L2_flux_boundary}
	Assume $\Gamma$ has the above properties and let $(\bm q, u)$ and $(\bm q_h, u_h)$ be the solution of \eqref{Poisson} and \eqref{HDG_discrete2}, respectively.  If  assumption (A) holds and $k\geqslant 1$, then we have:
	\begin{subequations}
		\begin{align}
		\| \bm q - \bm q_h \|_{\bm L^2(\Gamma)} &\le Ch^{k+1}(|\log h |^{1/2} +1 )(|\bm q|_{\bm W^{k+1,\infty}(\Omega)} + |u|_{ W^{k+1,\infty}(\Omega)}),\label{L_infty_q}\\
		\| u- u_h \|_{L^2(\Gamma)} &\le Ch^{k+1}(|\log h | +1 )(|\bm q|_{\bm W^{k+1,\infty}(\Omega)} + |u|_{ W^{k+1,\infty}(\Omega)}). \label{L_infty_u}
		\end{align}
		Furthermore,  we have the following error estimate for the postprocessed solution: 
		\begin{align}
		\| u- u_h^\star \|_{L^2(\Gamma)} \le Ch^{k+2}(|\log h | +1 )(|\bm q|_{\bm W^{k+1,\infty}(\Omega)} + |u|_{ W^{k+2,\infty}(\Omega)}), \label{L_infty_ustar}
		\end{align}
		where $u_h^\star $ is defined in \eqref{post_process_1}.
	\end{subequations}
\end{theorem}
\begin{remark}
	The result proves the observation seen in numerical experiments that the flux on $\Gamma$ converges at an optimal rate. The best theoretical estimates known to us before our paper is $\mathcal O(h^{k+1/2})$.
\end{remark}
\begin{proof}
	We only prove \eqref{L_infty_q} since the proof of \eqref{L_infty_u} and \eqref{L_infty_ustar}  are very similar. We define tubular neighborhoods of $\Gamma$ by 
	\begin{align*}
	\mathbb S_h:= \left\{K\in \mathcal T_h| \Gamma\cap \partial K\neq \emptyset\right\}.
	\end{align*}
	Then the number of the elements in $\mathbb S_h$ is order of $\mathcal O(h^{-1})$.
	\begin{align*}
	\| \bm q - \bm q_h \|_{L^2(\Gamma)}^2 &\le   \|\bm \Pi_V \bm q - \bm q \|_{L^2(\Gamma)}^2 + \|\bm \Pi_V \bm q - \bm q_h \|_{L^2(\Gamma)}^2  \\
	& \le \sum_{K\in \mathbb S_h}	\|\bm \Pi_V\bm q - \bm q\|_{\bm L^2(\partial K)}^2 +\|\bm \Pi_V \bm q - \bm q_h \|_{\bm L^2(\partial K)}^2\\
	& \le C\sum_{K\in \mathbb S_h}  \left( h^{2k+3} |\nabla^{k+1}\bm q|_{\bm L^\infty(\Omega)}^2 +   h^{-1}\|\bm \Pi_V \bm q - \bm q_h \|_{\bm L^2(K)}^2\right)&\textup{by } \eqref{error_pi_q1}\\
	&\le C\sum_{K\in \mathbb S_h}  \left(   h^{2k+3} |\nabla^{k+1}\bm q|_{\bm L^\infty(\Omega)}^2  +   h^{-1} h^2\|\bm \Pi_V \bm q - \bm q_h \|_{\bm L^\infty(K)}^2\right)&\textup{by } \eqref{inverse}\\
	&\le C\sum_{K\in \mathbb S_h}  h^{2k+3} |\nabla^{k+1}\bm q|_{\bm L^\infty(\Omega)}^2 + C  \|\bm \Pi_V \bm q - \bm q_h \|_{\bm L^\infty(\Omega)}^2\sum_{K\in \mathbb S_h}  h^{-1} h^2\\
	&\le C h^{2k+2} |\nabla^{k+1}\bm q|_{\bm L^\infty(\Omega)}^2  + C \|\bm \Pi_V \bm q - \bm q_h \|_{\bm L^\infty(\Omega)}^2\\
	&\le  Ch^{2k+2}(|\log h|^{1/2}+1)^2(|\bm q|_{\bm W^{k+1,\infty}(\Omega)} + |u|_{ W^{k+2,\infty}(\Omega)})^2&\textup{by } \eqref{prothmL_infty_q}.
	\end{align*}
	This completes the proof of \eqref{L_infty_q}.
\end{proof}

\section{Dirichlet Boundary Control Problem}\label{Dirichlet_Boundary_Control_Problem}
In this section, we consider {an elliptic Dirichlet boundary control problem.  Let $u_d\in L^2(\Omega)$ denote a given desired
	state for the solution, and let $\gamma>0$ be a given regularization parameter.  The problem is to solve the following optimization problem:}
\begin{subequations}
	\begin{align}
	\min_{{g\in L^2(\partial\Omega)}} J(g),  \quad  J(g):=\frac{1}{2}\|u-u_{d}\|^2_{L^{2}(\Omega)}+\frac{\gamma}{2}\|g\|^2_{L^{2}(\partial\Omega)}, \label{cost1}
	\end{align}
	where  $ u $ is the solution of the Poisson equation with non-homogeneous Dirichlet boundary conditions
	\begin{align}
	-\Delta u &= f \quad  \text{in} \; \Omega \label{D_p1},\\
	u &= g \quad  \text{on}\;  \partial\Omega. \label{D_p2}
	\end{align}
\end{subequations}
The function $g$ is called the control, and computing the optimal $g$ is the desired result of the above problem.

It is well known that the Dirichlet boundary control problem \eqref{cost1}-\eqref{p2} is equivalent to solving the following optimality system for $(u,z,g)$ given by:
\begin{subequations}\label{boundary_pro}
	\begin{align}
	-\Delta u &=f \quad  \quad  \quad \; \text{in} \; \Omega \label{boundary_pro_a},\\
	u&=g \quad \quad \quad\; \text{on}\;  \partial\Omega, \label{boundary_pro_u}\\
	-\Delta z  &=u-u_d \quad \; \text{in} \; \Omega \label{boundary_pro_b},\\
	z&=0 \quad \quad \quad\;\; \text{on}\;  \partial\Omega, \label{boundary_pro_c}\\
	g& =\gamma^{-1}\partial_{\bm n} z\quad \text{on}\;  \partial\Omega. \label{boundary_pro_d}
	\end{align}
\end{subequations}

Define $ \bm{q} = -\nabla u $ and $ \bm p= -\nabla z $, then the mixed weak form of  \eqref{boundary_pro_a}-\eqref{boundary_pro_d} is to find $(\bm q, u, \bm p, z,g)\in \bm H(\text{div},\Omega)\times L^2(\Omega)\times \bm H(\textup{div},\Omega) \times L^2(\Omega) \times L^2(\partial\Omega)$ such that
\begin{subequations}\label{mixed}
	\begin{align}
	(\bm{q},\bm{v}_1)-(u,\nabla\cdot\bm{r})+\left\langle g, \bm{v}_1\cdot \bm{n}\right\rangle&=0, \label{mixed_a}\\
	(\nabla\cdot \bm{q}, w_1) &= (f,w_1),  \label{mixed_b}\\
	(\bm{p},\bm{v}_2)-(z,\nabla\cdot\bm{v}_2)&=0, \label{mixed_c}\\
	(\nabla\cdot \bm{p}, w_2) -(u,w_2)&= -(u_d,w_2), \label{mixed_d}\\
	\left\langle\gamma g + \bm{p}\cdot \bm{n}, \xi \right\rangle &=0 \label{mixed_e}
	\end{align}
\end{subequations}
for all $(\bm{v}_1,w_1,\bm v_2,w_2,\xi)\in \bm H(\text{div},\Omega)\times L^2(\Omega)\times \bm H(\textup{div},\Omega) \times L^2(\Omega) \times L^2(\partial\Omega)$.

To give the HDG formulation of the above mixed system \eqref{mixed}, we need to introduce the following finite element space for the boundary control $g$:
\begin{align*}
\widehat W_h(\partial) := \{\widehat w_h \in L^2(\mathcal F_h^\partial): \widehat w_h|_{F}\in \mathcal{P}^k(F), \forall F\in \mathcal F_h^\partial\}.
\end{align*}

By the definition of $\mathscr B$ in \eqref{def_B} and setting $c=1$. The HDG formulation of the optimality system \eqref{HDG_discrete2} is to  find $({\bm{q}}_h,{\bm{p}}_h,u_h,z_h,\widehat u_h,\widehat z_h,g_h)\in \bm{V}_h \times\bm{V}_h\times W_h \times W_h\times \widehat W_h\times \widehat W_h\times\widehat W_h(\partial )$  such that
\begin{subequations}\label{HDG_full_discrete2}
	\begin{align}
	\mathscr B (\bm q_h,u_h,\widehat u_h;\bm v_1,w_1,\widehat w_1) &= -\langle g_h, \bm{v_1}\cdot \bm{n} + \tau w_1 \rangle_{{\mathcal F_h^\partial}}-(f, w_1)_{{\mathcal{T}_h}}, \label{HDG_full_discrete_a}\\
	\mathscr B(\bm p_h,z_h,\widehat z_h;\bm v_2,w_2,\widehat w_2) &= -(u_h -u_d,w_2)_{\mathcal T_h},\label{HDG_full_discrete_b}\\
	\gamma^{-1}\langle {\bm{p}}_h\cdot \bm{n} + \tau z_h, \widehat w_3\rangle_{{{\mathcal F_h^\partial}}} &= -\langle g_h, \widehat w_3 \rangle_{{\mathcal F_h^\partial}} \label{HDG_full_discrete_e}
	\end{align}
	for all $\left(\bm{v}_1, \bm{v}_2, w_1, w_2, \widehat w_1, \widehat w_2, \widehat w_3 \right)\in \bm{V}_h \times\bm{V}_h\times W_h \times W_h\times \widehat W_h\times \widehat W_h\times\widehat W_h(\partial ) $.
\end{subequations}

We can now state our main result of this section:
\begin{theorem}\label{main_result_optimal_control}
	Let $(\bm q, u, \bm p, z, g)$ and $({\bm{q}}_h,u_h, {\bm{p}}_h,z_h,g_h)$ be the solution of \eqref{mixed} and \eqref{HDG_full_discrete2}, respectively. If  assumption (A) holds and $k\geqslant 1$, then we have:
	\begin{align*}
	\|g - &g_h\|_{L^2(\partial\Omega)} +	\|u - u_h\|_{L^2(\Omega)} +  \|z - z_h\|_{L^2(\Omega)}    + \|\bm p - \bm p_h\|_{L^2(\Omega)}+ h^{1/2}\|\bm q - \bm q_h\|_{L^2(\Omega)}\\
	&\le Ch^{k+1}(|\log h|+1)(|\bm p|_{\bm W^{k+1,\infty}(\Omega)} + |z|_{ W^{k+1,\infty}(\Omega)}+|\bm q|_{\bm H^{k+1}(\Omega)} + |u|_{ H^{k+1}(\Omega)}).
	\end{align*}	
\end{theorem}

\begin{remark}
	Numerical experiments for the Dirichlet boundary control problem given in  \Cref{boundary_pro} always show  optimal order convergence rates if the solution is smooth enough. The first work to prove {this observation} can be found in \cite{May_Rannacher_Vexler_High_SICON_2013} by May, Rannacher and Vexler. The proof is based  on a duality argument and gives estimates for the {control	in weaker norms than $L^2(\partial \Omega)$. However, this technique is not straightforward for the HDG method, see \cite{MR3992054,HuShenSinglerZhangZheng_HDG_Dirichlet_control1,HuMateosSinglerZhangZhang1,HuMateosSinglerZhangZhang2,GongHuMateosSinglerZhang1}. Hence, \Cref{main_result_optimal_control} is the first proof that the HDG method achieves an optimal order convergence rate for the control, state and dual state,} provided we assume the solution of the Dirichlet boundary control problem is smooth enough.
	
\end{remark}

\subsection{Proof of \Cref{main_result_optimal_control}}

We follow the strategy in \cite{HuShenSinglerZhangZheng_HDG_Dirichlet_control1} and introduce an auxiliary problem: find $({\bm{q}}_h(g),{\bm{p}}_h(g),u_h(g),z_h(g),{\widehat{u}}_h(g),{\widehat{z}}_h(g))\in \bm{V}_h \times\bm{V}_h\times W_h \times W_h\times \widehat W_h\times \widehat W_h$ such that
\begin{subequations}\label{HDG_u}
	\begin{align}
	\mathscr B (\bm q_h(g),u_h(g),\widehat u_h(g);\bm v_1,w_1,\widehat w_1) &= -\langle  g, \bm{v_1}\cdot \bm{n}  +  \tau w_1 \rangle_{{\mathcal F_h^\partial}}-(f, w_1)_{{\mathcal{T}_h}}, \label{HDG_u_a}\\
	\mathscr B(\bm p_h(g),z_h(g),\widehat z_h (g);\bm v_2,w_2,\widehat w_2) &= -(u -u_d,w_2)_{\mathcal T_h}\label{HDG_u_b}
	\end{align}
\end{subequations}
for all $\left(\bm{v}_1, \bm{v}_2, w_1, w_2, \widehat w_1, \widehat w_2\right)\in \bm{V}_h \times\bm{V}_h\times W_h \times W_h\times \widehat W_h\times \widehat W_h$, where $g\in L^2(\partial\Omega)$ is the exact optimal control.

The proof now proceeds in three steps as follows.

\paragraph*{Step 1} We first bound the error between the solutions of the auxiliary problem \Cref{HDG_u} and the mixed weak form \eqref{mixed_a}-\eqref{mixed_e} of the optimality system.  The proof of \Cref{lemma_step1_conv_rates} can be found in  \cite[Theorem 4.1 Appendix]{Cockburn_Gopalakrishnan_Sayas_Porjection_MathComp_2010} and  \Cref{L2_flux_boundary}.
\begin{lemma}\label{lemma_step1_conv_rates}
	Let $(\bm q, u, \bm p, z, g)$ and $({\bm{q}}_h(g),{\bm{p}}_h(g),u_h(g),z_h(g),{\widehat{u}}_h(g),{\widehat{z}}_h(g))$ be the solution of \eqref{mixed} and \eqref{HDG_u}, respectively.  If  assumption (A) holds and $k\geqslant 1$, then we have:
	\begin{subequations}
		\begin{align}
		\|\bm q -\bm q_h(g)\|_{ \bm L^2(\Omega)}  &\le  Ch^{k+1}(|\bm q|_{\bm H^{k+1}(\Omega)} + |u|_{ H^{k+1}(\Omega)}),\label{error_au_step1_1}\\
		\|u -u_h(g)\|_{ L^2(\Omega)}  &\le  Ch^{k+1}(|\bm q|_{\bm H^{k+1}(\Omega)} + |u|_{ H^{k+1}(\Omega)}),\label{error_au_step1_2}\\
		\|\bm p -\bm p_h(g)\|_{ \bm L^2(\Omega)}  &\le  Ch^{k+1}(|\bm p|_{\bm H^{k+1}(\Omega)} + |z|_{ H^{k+1}(\Omega)}),\label{error_au_step1_12}\\
		\|z -z_h(g)\|_{ L^2(\Omega)}  &\le  Ch^{k+1}(|\bm p|_{\bm H^{k+1}(\Omega)} + |z|_{ H^{k+1}(\Omega)}),\label{error_au_step1_22}\\
		\norm {\bm p_h(g)-\bm p}_{\bm L^2(\partial\Omega)} & \le Ch^{k+1}(|\log h|^{1/2}+1)(|\bm p|_{\bm W^{k+1,\infty}(\Omega)} + |z|_{ W^{k+1,\infty}(\Omega)}),\label{error_au_step1_3}\\
		\|\tau (z_h(g) -z)\|_{ L^2(\partial\Omega)}& \le Ch^{k+1}(|\log h|+1)(|\bm p|_{\bm W^{k+1,\infty}(\Omega)} + |z|_{ W^{k+1,\infty}(\Omega)}).\label{error_au_step1_4}
		\end{align}
	\end{subequations}
\end{lemma}

\paragraph*{Step 2} 
Next,  we bound the error between the solutions of the auxiliary problem and the HDG problem \eqref{HDG_full_discrete}.  Note that
\begin{subequations}\label{eq_yh}
	\begin{align}
	\mathscr B(\bm q_h(g)-\bm q_h, u_h(g)- u_h,\widehat u_h(g)-\widehat u_h;\bm v_1, w_1,\widehat w_1) &= -\langle  g-g_h, \bm v_1\cdot \bm{n} + \tau w_1 \rangle_{{\mathcal F_h^\partial}}\label{eq_yh_yhu},\\
	\mathscr B(\bm p_h(g)-\bm p_h, z_h(g)- z_h,\widehat z_h(g)-\widehat z_h;\bm v_2, w_2,\widehat w_2) &= -(u-u_h, w_2)_{\mathcal T_h}\label{eq_zh_zhu}
	\end{align}
\end{subequations}
for all $\left(\bm{v}_1, \bm{v}_2, w_1, w_2, \widehat w_1, \widehat w_2\right)\in \bm{V}_h \times\bm{V}_h\times W_h \times W_h\times \widehat W_h\times \widehat W_h$.

\begin{lemma}\label{error_u_gg}
	Let $(u, g)$ and $(u_h, g_h)$ be the solution of \eqref{mixed} and \eqref{HDG_full_discrete2}, respectively.   If  assumption (A) holds and $k\geqslant 1$, then we have:
	\begin{align*}
	\norm{g-g_h}_{L^2(\partial\Omega)}  &\le Ch^{k+1}(|\log h|+1)(|\bm p|_{\bm W^{k+1,\infty}(\Omega)} + |z|_{ W^{k+1,\infty}(\Omega)}+|\bm q|_{\bm H^{k+1}(\Omega)} + |u|_{ H^{k+1}(\Omega)}),\\
	\| u - u_h\|_{L^2(\Omega)}  &\le Ch^{k+1}(|\log h|+1)(|\bm p|_{\bm W^{k+1,\infty}(\Omega)} + |z|_{ W^{k+1,\infty}(\Omega)}+|\bm q|_{\bm H^{k+1}(\Omega)} + |u|_{ H^{k+1}(\Omega)}).
	\end{align*}
\end{lemma}

\begin{proof}
	First, we  take  $(\bm v_1, w_1,\widehat w_1)  = (\bm p_h(g)-\bm p_h, z_h(g)- z_h,\widehat z_h(g)-\widehat z_h)$, $(\bm v_2, w_2,\widehat w_2)  = (\bm q_h(g)-\bm q_h, u_h(g)- u_h,\widehat u_h(g)-\widehat u_h)$ in \eqref{eq_yh},  and use \Cref{eq_B} to get 
	\begin{align*}
	\langle  g-g_h, (\bm p_h(g)-\bm p_h)\cdot \bm{n} +\tau  (z_h(g)- z_h) \rangle_{\mathcal F_h^\partial} =  ( u -u_h, u_h(g)- u_h)_{\mathcal{T}_h}.
	\end{align*}
	Since $g+\gamma^{-1}\bm p \cdot\bm n=0$ on $\mathcal F_h^\partial$ and $g_h+\gamma^{-1}\bm p_h\cdot\bm n +\gamma^{-1}\tau z_h=0$ on $\mathcal F_h^\partial$ we have
	\begin{align*}
	( u -u_h, u_h(g)- u_h)_{\mathcal{T}_h} &= \langle  g-g_h, \bm p_h(g)\cdot \bm{n} +\tau  z_h(g)  + \gamma g_h\rangle_{\mathcal F_h^\partial}\\
	& =  \langle  g-g_h, \bm p_h(g)\cdot \bm{n}  - \bm p\cdot\bm n + \bm p\cdot\bm n+\tau   z_h(g)  + \gamma g_h\rangle_{\mathcal F_h^\partial}\\
	& = \langle  g-g_h, \bm p_h(g)\cdot \bm{n}  - \bm p\cdot\bm n +\tau   z_h(g)  + \gamma g_h - \gamma g\rangle_{\mathcal F_h^\partial}\\
	& = \langle  g-g_h, \bm p_h(g)\cdot \bm{n}  - \bm p\cdot\bm n +\tau   z_h(g) \rangle_{\mathcal F_h^\partial} - \gamma\|g-g_h\|_{L^2(\partial\Omega)}^2.
	\end{align*}
	Since $z=0$ on $\mathcal F_h^\partial$,  we rearrange the above equality and obtain
	\begin{align*}
	\hspace{1em}&\hspace{-1em}  \gamma \norm{g-g_h}_{L^2(\partial\Omega)}^2  +  \|u - u_h\|_{L^2(\Omega)}^2 \\
	&=\langle (\bm p_h(g)-\bm p)\cdot\bm n +\tau z_h(g), g-g_h\rangle_{\mathcal F_h^\partial} -  ( u -u_h,u_h(g) - u )_{\mathcal{T}_h}\\
	&\le  \left(\norm {\bm p_h(g)-\bm p}_{\bm L^2(\partial\Omega)} +\|\tau (z_h(g)-z) \|_{ L^2(\partial\Omega)} \right)\norm{g-g_h}_{L^2(\partial\Omega)} \\
	&\quad + \|u-u_h\|_{L^2(\Omega)} \|u_h(g) - u \|_{L^2(\Omega)} .
	\end{align*}
	Our desired result follows by  Young's inequality,  the triangle inequality and \Cref{lemma_step1_conv_rates}.
\end{proof}

\paragraph*{Step 3}
\begin{lemma}\label{error_u_g}
	Let $(\bm p, z)$ and $(\bm p_h, z_h)$ be the solution of \eqref{mixed} and \eqref{HDG_full_discrete2}, respectively.  If  assumption (A) holds and $k\geqslant 1$, then we have:
	\begin{align*}
	\| \bm p- \bm p_h\|_{L^2(\Omega)}  &\le Ch^{k+1}(|\log h|+1)(|\bm p|_{\bm W^{k+1,\infty}(\Omega)} + |z|_{ W^{k+1,\infty}(\Omega)}+|\bm q|_{\bm H^{k+1}(\Omega)} + |u|_{ H^{k+1}(\Omega)}),\\
	\| z - z_h\|_{L^2(\Omega)}  &\le Ch^{k+1}(|\log h|+1)(|\bm p|_{\bm W^{k+1,\infty}(\Omega)} + |z|_{ W^{k+1,\infty}(\Omega)}+|\bm q|_{\bm H^{k+1}(\Omega)} + |u|_{ H^{k+1}(\Omega)}).
	\end{align*}
\end{lemma}

\begin{proof}
	By  \Cref{energy_norm} and letting $ (\bm v_2,w_2,\widehat w_2) = (\bm p_h(g)-\bm p_h, z_h(g)- z_h,\widehat z_h(g)-\widehat z_h)$ in the error equation \eqref{eq_zh_zhu}, we have
	\begin{align}\label{err_ph_phu_inter}
	\begin{split}
	\|\bm p_h(g)-\bm p_h\|^2_{L^2(\Omega)} & \le  - (u-u_h, z_h(g)- z_h)_{\mathcal T_h}\\
	&\le\|u-u_h\|_{\mathcal T_h}\|z_h(g)- z_h\|_{\mathcal T_h}\\
	&\le \frac 1 \rho\|u-u_h\|^2_{\mathcal T_h}+\rho\|z_h(g)- z_h\|_{\mathcal T_h}^2.
	\end{split}
	\end{align}
	Here $\rho$ is a positive constant which will be assigned later. 
	Next, we introduce the dual problem of finding $(\bm \Phi,\Psi)$ such that 
	\begin{equation}\label{Dual_PDE}
	\begin{split}
	c\bm{\Phi}+\nabla\Psi &= 0\qquad\qquad\qquad\qquad\text{in}\ \Omega,\\
	\nabla\cdot\bm \Phi &= z_h(g)- z_h\qquad\quad~~~\text{in}\ \Omega,\\
	\Psi &= 0\qquad\qquad\qquad\qquad\text{on}\ \partial\Omega.
	\end{split}
	\end{equation}
	Since the domain $\Omega$ is convex, we have the following regularity estimate
	\begin{align}\label{reg_e}
	\norm{\bm\Phi}_{\bm H^1(\Omega)} + \norm{\Psi}_{H^2(\Omega)} \le C_{\textup{reg}} \norm{z_h(g)- z_h}_{L^2(\Omega)}.
	\end{align}
	On the one hand, we take  $(\bm v_2,w_2,\widehat w_2) = (\bm\Pi_V\bm{\Phi},\Pi_W\Psi, \Pi_k^\partial\Psi)$ in \eqref{eq_zh_zhu} to get
	\begin{align}\label{one_e}
	\begin{split}
	\mathscr B &(\bm p_h(g)-\bm p_h, z_h(g)- z_h,\widehat z_h(g)-\widehat z_h;\bm\Pi_V\bm{\Phi},\Pi_W\Psi, \Pi_k^\partial\Psi)\\
	&=\mathscr B (\bm\Pi_V\bm{\Phi},\Pi_W\Psi, \Pi_k^\partial\Psi; \bm p_h(g)-\bm p_h, z_h(g)- z_h,\widehat z_h(g)-\widehat z_h)\\
	&=(\bm \Pi_V \bm \Phi - \bm \Phi,\bm p_h(g)-\bm p_h)_{\mathcal{T}_h} + \norm{z_h(g)- z_h}_{L^2(\Omega)}^2.
	\end{split}
	\end{align}
	On the other hand, by the error equation \eqref{eq_zh_zhu}, we have
	\begin{align}\label{two_e}
	&\mathscr B(\bm p_h(g)-\bm p_h, z_h(g)- z_h,\widehat z_h(g)-\widehat z_h;\bm\Pi_V\bm{\Phi},\Pi_W\Psi, \Pi_k^\partial\Psi) = - (u-u_h, \Pi_W\Psi)_{\mathcal T_h}.
	\end{align}
	Comparing the above two equalities \eqref{one_e}, \eqref{two_e} and \eqref{Proerr_q} gives
	\begin{align*}
	\norm{z_h(g)- z_h}_{L^2(\Omega)}^2
	&= - (u-u_h, \Pi_W\Psi)_{\mathcal T_h} -(\bm \Pi_V \bm \Phi - \bm \Phi,\bm p_h(g)-\bm p_h)_{\mathcal{T}_h}\\
	&\le C\| u -u_h\|_{L^2(\Omega)}^2 + \frac 1 4 \norm{z_h(g)- z_h}_{L^2(\Omega)}^2 \\
	&\quad + \frac {h^2}{4} \norm{z_h(g)- z_h}_{L^2(\Omega)}^2 + C\|\bm p_h(g)-\bm p_h\|^2_{L^2(\Omega)}\\
	&\le  C\left(1+\frac{1}{\rho}\right)\| u -u_h\|_{L^2(\Omega)}^2 + \frac 1 2 \norm{z_h(g)- z_h}_{L^2(\Omega)}^2 + C\rho \norm{z_h(g)- z_h}_{L^2(\Omega)}^2.
	\end{align*}
	Taking $\rho = \dfrac{1}{4C}$, then we have 
	\begin{align}\label{error_z_dual}
	\norm{z_h(g)- z_h}_{L^2(\Omega)}\le C \| u -u_h\|_{L^2(\Omega)}.
	\end{align}
	Inserting this inequality  into \eqref{err_ph_phu_inter} gives
	\begin{align}\label{error_p_dual}
	\|\bm p_h(g)-\bm p_h\|_{\bm L^2(\Omega)} \le C \| u -u_h\|_{L^2(\Omega)}.
	\end{align}
	Then our desired result follows by \eqref{error_z_dual}, \eqref{error_p_dual} and \Cref{error_u_g}.
\end{proof}

\begin{lemma}\label{error_q}
	Let $(\bm q, u, \bm p, z, g)$ and $({\bm{q}}_h,u_h, {\bm{p}}_h,z_h,g_h)$ be the solution of \eqref{mixed} and \eqref{HDG_full_discrete2}, respectively.   If  assumption (A) holds and $k\geqslant 1$, then we have:
	\begin{align*}
	\| \bm q - \bm q_h\|_{L^2(\Omega)}  &\le Ch^{k+\frac 1 2}(|\log h|+1)(|\bm p|_{\bm W^{k+1,\infty}(\Omega)} + |z|_{ W^{k+1,\infty}(\Omega)}+|\bm q|_{\bm H^{k+1}(\Omega)} + |u|_{ H^{k+1}(\Omega)}).
	\end{align*}
\end{lemma}
\begin{proof}
	On the one hand, by the error equation  \eqref{eq_yh_yhu}, we have
	\begin{align*}
	\hspace{1em}&\hspace{-1em}\mathscr B(\bm q_h(g)-\bm q_h, u_h(g)- u_h,\widehat u_h(g)-\widehat u_h;\bm q_h(g)-\bm q_h, u_h(g)- u_h,\widehat u_h(g)-\widehat u_h)\\
	&= -\langle g-g_h,(\bm q_h(g)-\bm q_h)\cdot \bm n-\tau (u_h(g)- u_h)\rangle_{\mathcal F_h}\\
	&\le \norm {g-g_h}_{L^2(\partial\Omega)} (\norm {\bm q_h(g)-\bm q_h}_{L^2(\partial\Omega)} + \norm { u_h(g)- u_h}_{L^2(\partial\Omega)})\\
	&\le Ch^{-1/2}\norm {g-g_h}_{L^2(\partial\Omega)} (\norm {\bm q_h(g)-\bm q_h}_{L^2(\Omega)} + \norm { u_h(g)- u_h}_{L^2(\Omega)}).
	\end{align*}
	On the other hand, by  \Cref{energy_norm}, we obtain
	\begin{align*}
	\hspace{1em}&\hspace{-1em} \mathscr B(\bm q_h(g)-\bm q_h, u_h(g)- u_h,\widehat u_h(g)-\widehat u_h;\bm q_h(g)-\bm q_h, u_h(g)- u_h,\widehat u_h(g)-\widehat u_h)\\
	& =  \|\bm q_h(g)-\bm q_h\|_{{\mathcal{T}_h}}^2 +\|\sqrt{\tau} ((u_h(g)- u_h)-(\widehat u_h(g)-\widehat u_h))\|_{\partial\mathcal T_h}.
	\end{align*}
	Comparing the above two inequalities,  using Young's inequality and \eqref{error_u_gg} gives 
	\begin{align}
	\|\bm q_h(g)-\bm q_h\|_{{\mathcal{T}_h}}\le Ch^{-1/2}\norm {g-g_h}_{L^2(\partial\Omega)}.
	\end{align}
	Then our desired result follows by \Cref{error_u_g}, the triangle inequality and \eqref{error_au_step1_1}.
\end{proof}

\section{Numerical Results}
\label{sec:numerics}

In this section, we present two examples to illustrate our theoretical results.

\begin{example}\label{example_1}
	We  first  test the convergence rate of the $L^\infty$ norm estimate on a convex domain and the $L^2$ norm estimate on the boundary.
	The data is chosen to be
	\begin{gather*}
	\Omega = (0,1)\times (0,1), \quad c = 1,\quad u(x, y) = \sin(10 x).
	\end{gather*}
	The source term $f$ is  chosen to match the exact solution of  \Cref{Poisson} and the approximation errors  are listed in \Cref{table_1} for the $L^\infty(\Omega)$ norm error and \Cref{table_2} for the $L^2(\partial\Omega)$ norm error.  The rates  match the theoretical predictions in \Cref{main_result_Linfty_norm,L2_flux_boundary}.
	
	Our theoretical result needs the domain to be convex, but it is interesting to observe whether the convergence rate can still hold for a non-convex domain. For example, we choose the same data as above except the domain is chosen to be an $L$-shape domain:
	\begin{gather*}
	\Omega = (0,1)\times (0,1) \backslash [1/2,1)\times (0,1/2].
	\end{gather*}
	{In this case the $H^2$ regularity of $\Psi_1$ and $\Psi_2$ in \Cref{regularity_Green_function} does not hold.}
	The approximation errors  are listed in \Cref{table_3} for the $L^\infty(\Omega)$ norm  error (the $L^2(\partial\Omega)$ norm error
	also converges at the quasi-optimal rate: results are not shown). It is obvious that the quasi-optimal convergence rate is still seen for the $L$-shape domain. 

	\begin{table}
		{
			\begin{tabular}{c|c|c|c|c|c|c|c}
				\Xhline{1pt}

				\multirow{2}{*}{Degree}
				&\multirow{2}{*}{$\frac{h}{\sqrt{2}}$}	
				&\multicolumn{2}{c|}{$\|\bm{q}-\bm{q}_h\|_{\bm L^\infty(\Omega)}$}	
				&\multicolumn{2}{c|}{$\|u-u_h\|_{L^\infty(\Omega)}$}	
				&\multicolumn{2}{c}{$\|u-u_h^\star\|_{L^\infty(\Omega)}$}	\\
				\cline{3-8}
				& &Error &Rate
				&Error &Rate
				&Error &Rate
				\\
				\cline{1-8}
				\multirow{5}{*}{ $k=1$}
				&	$2^{-1}$	&	 1.8881E+01	&	-	&	  8.1191E+00	&	-	&	 1.4941E+00	&	-	 \\ 
				&	$2^{-2}$	&	 1.0384E+01	&	0.86	&	  2.9595E+00	&	1.46	&	 4.0248E-01	&	1.89	 \\ 
				&	$2^{-3}$	&	 2.9862E+00	&	1.80	&	  7.7800E-01	&	1.93	&	 5.9420E-02	&	2.76	 \\ 
				&	$2^{-4}$	&	 7.5737E-01	&	1.98	&	  2.0046E-01	&	1.96	&	 7.6187E-03	&	2.96	 \\ 
				&	$2^{-5}$	&	 1.9487E-01	&	1.96	&	  4.9683E-02	&	2.01	&	 9.7985E-04	&	2.96	 \\

				\cline{1-8}
				\multirow{5}{*}{$k=2$}
				&	$2^{-1}$	&	 1.8115E+01	&	-	&	  6.5763E+00	&	-	&	 7.2961E-01	&	-	 \\ 
				&	$2^{-2}$	&	 3.4370E+00	&	2.40	&	  1.0994E+00	&	2.58	&	 6.9452E-02	&	3.39	 \\ 
				&	$2^{-3}$	&	 4.7355E-01	&	2.86	&	  1.4548E-01	&	2.92	&	 4.6990E-03	&	3.89	 \\ 
				&	$2^{-4}$	&	 6.2699E-02	&	2.92	&	  1.7948E-02	&	3.02	&	 3.1054E-04	&	3.92	 \\ 
				&	$2^{-5}$	&	 7.8798E-03	&	2.99	&	  2.2918E-03	&	2.97	&	 1.9522E-05	&	3.99	 \\

				\Xhline{1pt}

			\end{tabular}
		}
		\caption{\Cref{example_1}: $L^\infty(\Omega)$ errors for $\bm{q}_h$, $u_h$ and $u_h^\star$ on the convex domain $(0,1)\times (0,1)$.}\label{table_1}
	\end{table}

	\begin{table}
		{
			\begin{tabular}{c|c|c|c|c|c|c|c}
				\Xhline{1pt}

				\multirow{2}{*}{Degree}
				&\multirow{2}{*}{$\frac{h}{\sqrt{2}}$}	
				&\multicolumn{2}{c|}{$\|\bm{q}-\bm{q}_h\|_{\bm L^2(\partial\Omega)}$}	
				&\multicolumn{2}{c|}{$\|u-u_h\|_{L^2(\partial\Omega)}$}	
				&\multicolumn{2}{c}{$\|u-u_h^\star\|_{L^2(\partial\Omega)}$}	\\
				\cline{3-8}
				& &Error &Rate
				&Error &Rate
				&Error &Rate
				\\
				\cline{1-8}
				\multirow{5}{*}{ $k=1$}
				&	$2^{-1}$	&	 9.3751E+00	&	-	&	  3.9706E+00	&	-	&	 5.4467E-01	&	-	 \\ 
				&	$2^{-2}$	&	 4.1197E+00	&	1.19	&	  1.9143E+00	&	1.05	&	 1.0446E-01	&	2.38	 \\ 
				&	$2^{-3}$	&	 1.1791E+00	&	1.80	&	  6.1659E-01	&	1.63	&	 1.4777E-02	&	2.82	 \\ 
				&	$2^{-4}$	&	 3.0648E-01	&	1.94	&	  1.6398E-01	&	1.91	&	 1.9370E-03	&	2.93	 \\ 
				&	$2^{-5}$	&	 7.7039E-02	&	1.99	&	  4.1472E-02	&	1.98	&	 2.4450E-04	&	2.99	 \\

				\cline{1-8}
				\multirow{5}{*}{$k=2$}
				&	$2^{-1}$	&	 6.4399E+00	&	-	&	  3.4609E+00	&	-	&	 1.9906E-01	&	-	 \\ 
				&	$2^{-2}$	&	 9.3121E-01	&	2.79	&	  5.3204E-01	&	2.70	&	 1.3075E-02	&	3.93	 \\ 
				&	$2^{-3}$	&	 1.1602E-01	&	3.00	&	  5.7436E-02	&	3.21	&	 9.1221E-04	&	3.84	 \\ 
				&	$2^{-4}$	&	 1.4279E-02	&	3.02	&	  6.5665E-03	&	3.13	&	 5.8411E-05	&	3.97	 \\ 
				&	$2^{-5}$	&	 1.7866E-03	&	3.00	&	  8.0200E-04	&	3.03	&	 3.6752E-06	&	3.99	 \\

				\Xhline{1pt}

			\end{tabular}
		}
		\caption{\Cref{example_1}: $L^2(\partial \Omega)$ errors for $\bm{q}_h$, $u_h$ and $u_h^\star$ on the convex  domain $(0,1)\times (0,1)$.}\label{table_2}
	\end{table}

	\begin{table}
		{
			\begin{tabular}{c|c|c|c|c|c|c|c}
				\Xhline{1pt}

				\multirow{2}{*}{Degree}
				&\multirow{2}{*}{$\frac{h}{\sqrt{2}}$}	
				&\multicolumn{2}{c|}{$\|\bm{q}-\bm{q}_h\|_{\bm L^\infty(\Omega)}$}	
				&\multicolumn{2}{c|}{$\|u-u_h\|_{L^\infty(\Omega)}$}	
				&\multicolumn{2}{c}{$\|u-u_h^\star\|_{L^\infty(\Omega)}$}	\\
				\cline{3-8}
				& &Error &Rate
				&Error &Rate
				&Error &Rate
				\\
				\cline{1-8}
				\multirow{5}{*}{ $k=1$}
				&	$2^{-1}$	&	 1.9604E+01	&	-	&	  8.1190E+00	&	-	&	 1.4713E+00	&	-	 \\ 
				&	$2^{-2}$	&	 9.9832E+00	&	0.97	&	  2.9608E+00	&	1.46	&	 3.7109E-01	&	1.99	 \\ 
				&	$2^{-3}$	&	 2.9810E+00	&	1.74	&	  7.7748E-01	&	1.93	&	 5.9410E-02	&	2.64	 \\ 
				&	$2^{-4}$	&	 7.5727E-01	&	1.98	&	  2.0046E-01	&	1.96	&	 7.6187E-03	&	2.96	 \\ 
				&	$2^{-5}$	&	 1.9487E-01	&	1.96	&	  4.9683E-02	&	2.01	&	 9.8015E-04	&	2.96	 \\

				\cline{1-8}
				\multirow{5}{*}{$k=2$}
				&	$2^{-1}$	&	 1.6115E+01	&	-	&	  6.4608E+00	&	-	&	 5.6157E-01	&	-	 \\ 
				&	$2^{-2}$	&	 3.4372E+00	&	2.23	&	  1.0994E+00	&	2.55	&	 6.9454E-02	&	3.02	 \\ 
				&	$2^{-3}$	&	 4.7348E-01	&	2.86	&	  1.4548E-01	&	2.92	&	 4.7007E-03	&	3.89	 \\ 
				&	$2^{-4}$	&	 6.2862E-02	&	2.91	&	  1.7948E-02	&	3.02	&	 3.1283E-04	&	3.91	 \\ 
				&	$2^{-5}$	&	 7.8980E-03	&	2.99	&	  2.2918E-03	&	2.97	&	 1.9653E-05	&	3.99	 \\

				\Xhline{1pt}

			\end{tabular}
		}
		\caption{\Cref{example_1}: $L^\infty(\Omega)$ errors for $\bm{q}_h$, $u_h$ and $u_h^\star$ on the nonconvex L-shaped domain.}\label{table_3}
	\end{table}

	%
	%
	%
	%
	%
	%
	%
	%
	%
	%
	%

\end{example}

\begin{example}\label{example_2}
	Lastly, we  test the convergence rate for a smooth solution to the Dirichlet boundary control problem.
	The data and the exact solution is chosen to be
	\begin{gather*}
	\Omega = (0,1)\times (0,1),\quad \gamma = 1,\quad   u(x, y) = -\pi (\sin(\pi x) + \sin(\pi y)),\quad z(x, y) = \sin(\pi x)\sin(\pi y).
	\end{gather*}
	The source term $f$, the desired state $u_d$  and the control $g$ are  chosen to match the exact solution of  \Cref{boundary_pro} and the approximation errors  are listed in \Cref{table_5} when $k=1$.  {Results (not shown) for $k=2$ also confirm
		the predicted higher order convergence rate in this case.}
	The rates are matched with \Cref{main_result_optimal_control}.
	\begin{table}
		\begin{center}
			\begin{tabular}{cccccc|c}
				\hline
				${h}/{\sqrt 2}$ &1/16& 1/32&1/64 &1/128 & 1/256    &\textup{EO} \\
				\hline
				$\|\bm q - \bm q_h\|_{\bm L^2(\Omega)}$&2.1856E-02 & 6.3683E-03  &1.9677E-03	   &6.3980E-04	 &2.1568E-04\\
				order&- &1.78  &1.69& 1.62&1.57&1.50\\
				\hline
				$\|\bm p - \bm p_h\|_{\bm L^2(\Omega)}$ &6.3866E-03	   &1.5958E-03	   &3.9873E-04   &9.9650E-05	 &2.4911E-05	
				\\
				order&-& 2.00 &2.00  &2.00   &2.00 &2.00 \\
				\hline
				$\|u - u_h\|_{ L^2(\Omega)}$ &8.3560E-03   &2.1051E-03	   &5.2796E-04	   & 1.3218E-04   &3.3073E-05\\
				order&-& 1.99 &2.00  &2.00   &2.00 &2.00 \\
				\hline
				$\| z-  z_h\|_{ L^2(\Omega)}$&3.1536E-03	   & 7.9650E-04   &2.0006E-04	   & 5.0125E-05	   & 1.2545E-05 \\
				order&-& 1.99 &1.99  &2.00   &2.00 &2.00 \\
				\hline
				$\norm{g-g_h}_{L^2(\partial\Omega)}$&7.2110E-03	   & 1.8119E-03   &4.5412E-04   & 1.1367E-04   & 2.8425E-05\\
				order&-& 1.99 &2.00  &2.00   &2.00 &2.00 \\
				\hline
			\end{tabular}
		\end{center}
		\caption{\Cref{example_2}, $k=1$: Errors, observed convergence orders, and expected order (EO) for the control $g$, state $u$, adjoint state $ z$,  and their fluxes $\bm q$ and $\bm p$.}\label{table_5}
	\end{table}
	%
	
\end{example}

\section{Conclusion}
We have proved  quasi-optimal $L^\infty$ norm estimates for the Poisson equation in 2D. Using this result, we obtained quasi-optimal $L^2$ estimates on an interface. Moreover, we  obtained   quasi-optimal convergence rates for  the Dirichlet boundary control of Poisson's  equation, provided the solution is smooth enough.

Our work suggests several interesting directions for further research. First we would like to extend the results
to cover  $L^\infty$ norm estimates in 3D.  In addition the quasi-uniformity assumption on our mesh is restrictive for problems
that require adaptive mesh refinement, including those on non-convex domains. Finally it would be desirable
to prove the optimal convergence rate for the Dirichlet boundary control of PDEs without assuming that the solution is smooth.

\bibliographystyle{siamplain}
\bibliography{Maxwell,Linfity,Dirichlet_Boundary_Control,Mypapers,Added,HDG}

\begin{thebibliography}{10}

\bibitem{Apel_MCRF_2018}
{\sc T.~Apel, M.~Mateos, J.~Pfefferer, and A.~R\"{o}sch}, {\em Error estimates
  for {D}irichlet control problems in polygonal domains: quasi-uniform meshes},
  Math. Control Relat. Fields, 8 (2018), pp.~217--245,
  \url{https://doi.org/10.3934/mcrf.2018010}.

\bibitem{ChenCockburnSinglerZhang1}
{\sc G.~Chen, B.~Cockburn, J.~Singler, and Y.~Zhang}, {\em Superconvergent
  interpolatory {HDG} methods for reaction diffusion equations {I}: {A}n {${\rm
  HDG}_k$} method}, J. Sci. Comput., 81 (2019), pp.~2188--2212,
  \url{https://doi.org/10.1007/s10915-019-01081-3}.

\bibitem{ChenMonkZhang1}
{\sc G.~Chen, P.~Monk, and Y.~Zhang}, {\em Superconvergent {HDG} methods for
  {M}axwell's equations via the {M}-decomposition},  (2019).
\newblock https://arxiv.org/abs/1905.07383.

\bibitem{MR3992054}
{\sc G.~Chen, J.~R. Singler, and Y.~Zhang}, {\em An {HDG} method for
  {D}irichlet boundary control of convection dominated diffusion {PDE}s}, SIAM
  J. Numer. Anal., 57 (2019), pp.~1919--1946,
  \url{https://doi.org/10.1137/18M1208708}.

\bibitem{Chen_Math_Comp_2005}
{\sc H.~Chen}, {\em Pointwise error estimates of the local discontinuous
  {G}alerkin method for a second order elliptic problem}, Math. Comp., 74
  (2005), pp.~1097--1116, \url{https://doi.org/10.1090/S0025-5718-04-01700-4}.

\bibitem{Chen_SINUM_2006}
{\sc H.~Chen}, {\em Pointwise error estimates for finite element solutions of
  the {S}tokes problem}, SIAM J. Numer. Anal., 44 (2006), pp.~1--28,
  \url{https://doi.org/10.1137/S0036142903438100}.

\bibitem{Cockburn_JSC_2007}
{\sc B.~Cockburn and B.~Dong}, {\em An analysis of the minimal dissipation
  local discontinuous {G}alerkin method for convection-diffusion problems}, J.
  Sci. Comput., 32 (2007), pp.~233--262,
  \url{https://doi.org/10.1007/s10915-007-9130-3}.

\bibitem{Cockburn_Gopalakrishnan_Lazarov_Unify_SINUM_2009}
{\sc B.~Cockburn, J.~Gopalakrishnan, and R.~Lazarov}, {\em Unified
  hybridization of discontinuous {G}alerkin, mixed, and continuous {G}alerkin
  methods for second order elliptic problems}, SIAM J. Numer. Anal., 47 (2009),
  pp.~1319--1365, \url{https://doi.org/10.1137/070706616}.

\bibitem{Cockburn_Gopalakrishnan_Sayas_Porjection_MathComp_2010}
{\sc B.~Cockburn, J.~Gopalakrishnan, and F.-J. Sayas}, {\em A projection-based
  error analysis of {HDG} methods}, Math. Comp., 79 (2010), pp.~1351--1367,
  \url{https://doi.org/10.1090/S0025-5718-10-02334-3}.

\bibitem{Duran_M2AN_1988}
{\sc R.~Dur\'{a}n, R.~H. Nochetto, and J.~P. Wang}, {\em Sharp maximum norm
  error estimates for finite element approximations of the {S}tokes problem in
  {$2$}-{D}}, Math. Comp., 51 (1988), pp.~491--506,
  \url{https://doi.org/10.2307/2008760}.

\bibitem{Duran_M2AN_1988_nonlinear}
{\sc R.~G. Dur\'{a}n}, {\em Error analysis in {$L^p,\;1\leq p\leq\infty,$} for
  mixed finite element methods for linear and quasi-linear elliptic problems},
  RAIRO Mod\'{e}l. Math. Anal. Num\'{e}r., 22 (1988), pp.~371--387,
  \url{https://doi.org/10.1051/m2an/1988220303711}.

\bibitem{Gastaldi_NM_1987}
{\sc L.~Gastaldi and R.~Nochetto}, {\em Optimal {$L^\infty$}-error estimates
  for nonconforming and mixed finite element methods of lowest order}, Numer.
  Math., 50 (1987), pp.~587--611, \url{https://doi.org/10.1007/BF01408578}.

\bibitem{Gastaldi_M2AN_1989}
{\sc L.~Gastaldi and R.~H. Nochetto}, {\em Sharp maximum norm error estimates
  for general mixed finite element approximations to second order elliptic
  equations}, RAIRO Mod\'{e}l. Math. Anal. Num\'{e}r., 23 (1989), pp.~103--128,
  \url{https://doi.org/10.1051/m2an/1989230101031}.

\bibitem{Scott_NM_2015}
{\sc V.~Girault, R.~H. Nochetto, and L.~R. Scott}, {\em Max-norm estimates for
  {S}tokes and {N}avier-{S}tokes approximations in convex polyhedra}, Numer.
  Math., 131 (2015), pp.~771--822,
  \url{https://doi.org/10.1007/s00211-015-0707-8}.

\bibitem{Scott_JMPA_2005}
{\sc V.~Girault, R.~H. Nochetto, and R.~Scott}, {\em Maximum-norm stability of
  the finite element {S}tokes projection}, J. Math. Pures Appl. (9), 84 (2005),
  pp.~279--330, \url{https://doi.org/10.1016/j.matpur.2004.09.017}.

\bibitem{HuMateosSinglerZhangZhang2}
{\sc W.~Gong, W.~Hu, M.~Mateos, J.~Singler, X.~Zhang, and Y.~Zhang}, {\em A
  {N}ew {HDG} {M}ethod for {D}irichlet {B}oundary {C}ontrol of {C}onvection
  {D}iffusion {PDE}s {II}: {L}ow {R}egularity}, SIAM J. Numer. Anal., 56
  (2018), pp.~2262--2287, \url{https://doi.org/10.1137/17M1152103}.

\bibitem{GongHuMateosSinglerZhang1}
{\sc W.~Gong, W.~Hu, M.~Mateos, J.~Singler, and Y.~Zhang}, {\em {An} {HDG}
  {M}ethod for {T}angential {D}irichlet {B}oundary {C}ontrol of {S}tokes
  {E}quations {I}: {H}igh {R}egularity, https://arxiv.org/abs/1811.08522},
  (2019).

\bibitem{Guzman_Math_Comp_2008}
{\sc J.~Guzm\'{a}n}, {\em Local and pointwise error estimates of the local
  discontinuous {G}alerkin method applied to the {S}tokes problem}, Math.
  Comp., 77 (2008), pp.~1293--1322,
  \url{https://doi.org/10.1090/S0025-5718-08-02067-X}.

\bibitem{Guzman_Math_Comp_2012}
{\sc J.~Guzm\'{a}n and D.~Leykekhman}, {\em Pointwise error estimates of finite
  element approximations to the {S}tokes problem on convex polyhedra}, Math.
  Comp., 81 (2012), pp.~1879--1902,
  \url{https://doi.org/10.1090/S0025-5718-2012-02603-2}.

\bibitem{Horger_CVS_2013}
{\sc T.~Horger, J.~M. Melenk, and B.~Wohlmuth}, {\em On optimal {$L^2$}- and
  surface flux convergence in {FEM}}, Comput. Vis. Sci., 16 (2013),
  pp.~231--246, \url{https://doi.org/10.1007/s00791-015-0237-z}.

\bibitem{HuMateosSinglerZhangZhang1}
{\sc W.~Hu, M.~Mateos, J.~Singler, and Y.~Zhang}, {\em {A} {N}ew {HDG} {M}ethod
  for {D}irichlet {B}oundary {C}ontrol of {C}onvection {D}iffusion {PDE}s {I}:
  High regularity, https://arxiv.org/abs/1801.01461}.

\bibitem{HuShenSinglerZhangZheng_HDG_Dirichlet_control1}
{\sc W.~Hu, J.~Shen, J.~R. Singler, Y.~Zhang, and X.~Zheng}, {\em A
  superconvergent hybridizable discontinuous {G}alerkin method for {D}irichlet
  boundary control of elliptic {PDE}s}, Numer. Math., 144 (2020), pp.~375--411,
  \url{https://doi.org/10.1007/s00211-019-01090-2}.

\bibitem{May_Rannacher_Vexler_High_SICON_2013}
{\sc S.~May, R.~Rannacher, and B.~Vexler}, {\em Error analysis for a finite
  element approximation of elliptic {D}irichlet boundary control problems},
  SIAM J. Control Optim., 51 (2013), pp.~2585--2611,
  \url{https://doi.org/10.1137/080735734}.

\bibitem{Melenk_IMA_2014}
{\sc J.~M. Melenk, H.~Rezaijafari, and B.~Wohlmuth}, {\em Quasi-optimal {\it a
  priori} estimates for fluxes in mixed finite element methods and an
  application to the {S}tokes-{D}arcy coupling}, IMA J. Numer. Anal., 34
  (2014), pp.~1--27, \url{https://doi.org/10.1093/imanum/drs048}.

\bibitem{Melenk_SINUM_2012}
{\sc J.~M. Melenk and B.~Wohlmuth}, {\em Quasi-optimal approximation of surface
  based {L}agrange multipliers in finite element methods}, SIAM J. Numer.
  Anal., 50 (2012), pp.~2064--2087, \url{https://doi.org/10.1137/110832999}.

\bibitem{Pfefferer_SINUM_2019}
{\sc J.~Pfefferer and M.~Winkler}, {\em Finite element error estimates for
  normal derivatives on boundary concentrated meshes}, SIAM J. Numer. Anal., 57
  (2019), pp.~2043--2073, \url{https://doi.org/10.1137/18M1181341}.

\bibitem{Schatz_Math_Comp_1977}
{\sc A.~H. Schatz and L.~B. Wahlbin}, {\em Interior maximum norm estimates for
  finite element methods}, Math. Comp., 31 (1977), pp.~414--442,
  \url{https://doi.org/10.2307/2006424}.

\bibitem{Schatz_Math_Comp_1982}
{\sc A.~H. Schatz and L.~B. Wahlbin}, {\em On the quasi-optimality in
  {$L\sb{\infty }$} of the {$\dot H\sp{1}$}-projection into finite element
  spaces}, Math. Comp., 38 (1982), pp.~1--22,
  \url{https://doi.org/10.2307/2007461}.

\bibitem{Schatz_Math_Comp_1995}
{\sc A.~H. Schatz and L.~B. Wahlbin}, {\em Interior maximum-norm estimates for
  finite element methods. {II}}, Math. Comp., 64 (1995), pp.~907--928,
  \url{https://doi.org/10.2307/2153476}.

\bibitem{Scholz_M2AN_1977}
{\sc R.~Scholz}, {\em {$L\sb{\infty }$}-convergence of saddle-point
  approximations for second order problems}, RAIRO Anal. Num\'{e}r., 11 (1977),
  pp.~209--216, 221, \url{https://doi.org/10.1051/m2an/1977110202091}.

\bibitem{Scott_Math_Comp_1976}
{\sc R.~Scott}, {\em Optimal {$L\sp{\infty }$} estimates for the finite element
  method on irregular meshes}, Math. Comp., 30 (1976), pp.~681--697,
  \url{https://doi.org/10.2307/2005390}.

\bibitem{book2}
{\sc Z.~Shi and M.~Wang}, {\em {F}inite {E}lement {M}ethods}, vol.~58 of Series
  in Information and Computation Science, Science Press, Beijing, 2013.

\bibitem{wang1989asymptotic}
{\sc J.~P. Wang}, {\em Asymptotic expansions and {$L^\infty$}-error estimates
  for mixed finite element methods for second order elliptic problems}, Numer.
  Math., 55 (1989), pp.~401--430, \url{https://doi.org/10.1007/BF01396046}.

\bibitem{Winkler_NM_2020}
{\sc M.~Winkler}, {\em Error estimates for variational normal derivatives and
  {D}irichlet control problems with energy regularization}, Numer. Math., 144
  (2020), pp.~413--445, \url{https://doi.org/10.1007/s00211-019-01091-1}.

\end{thebibliography}

\end{document}